\documentclass[11pt,a4paper,reqno]{amsart}
\usepackage[english]{babel}
\usepackage[applemac]{inputenc}
\usepackage[T1]{fontenc}
\usepackage{palatino}
\usepackage{amsmath}
\usepackage{amssymb}
\usepackage{amsthm}
\usepackage{amsfonts}
\usepackage{graphicx}
\usepackage{color}
\usepackage{vmargin}
\usepackage{mathtools}
\usepackage{stix}

\usepackage{pgf,tikz}
\usepackage{subfig}
\usepackage{mathrsfs}
\usetikzlibrary{arrows}

\usepackage{pgfplots}
\pgfplotsset{compat=1.14}
\usepackage{pstricks-add}

\usetikzlibrary[patterns]

\setmarginsrb{20mm}{20mm}{20mm}{20mm}{10mm}{10mm}{10mm}{10mm}

\usepackage[colorlinks = true, citecolor = black]{hyperref}
\newtheorem{theorem}{Theorem}[section]

\newtheorem{cor}[theorem]{Corollary}
\newtheorem{prop}[theorem]{Proposition}

\newtheorem{lem}[theorem]{Lemma}

\theoremstyle{definition}
\newtheorem{defn}{Definition}[section]

\newtheorem{remark}[theorem]{Remark}

\newtheorem*{theorem*}{Theorem}

\def\vint{\mathop{\mathchoice%
          {\setbox0\hbox{$\displaystyle\intop$}\kern 0.22\wd0%
           \vcenter{\hrule width 0.6\wd0}\kern -0.82\wd0}%
          {\setbox0\hbox{$\textstyle\intop$}\kern 0.2\wd0%
           \vcenter{\hrule width 0.6\wd0}\kern -0.8\wd0}%
          {\setbox0\hbox{$\scriptstyle\intop$}\kern 0.2\wd0%
           \vcenter{\hrule width 0.6\wd0}\kern -0.8\wd0}%
          {\setbox0\hbox{$\scriptscriptstyle\intop$}\kern 0.2\wd0%
           \vcenter{\hrule width 0.6\wd0}\kern -0.8\wd0}}%
          \mathopen{}\int}

\newcommand{\R}{\mathbb{R}}
\newcommand{\Rn}{{\mathbb R}^n}

\newcommand{\ep}{\varepsilon}
\newcommand{\Om}{\Omega}
\newcommand{\om}{\omega}
\newcommand{\osc}{{\rm{osc}}}
\newcommand{\ud}{\mathrm {d}}

\newcommand{\diam}{\operatorname{diam}}

\definecolor{blau}{rgb}{0.1,0.0,0.9}
\definecolor{violet}{rgb}{0.54, 0.17, 0.89}

\newcommand{\kom}[1]{}
\renewcommand{\kom}[1]{{\bf \blue /#1/}}

\newcounter{komcounter}
\numberwithin{komcounter}{section}

\def\XXint#1#2#3{{\setbox0=\hbox{$#1{#2#3}{\int}$}
		\vcenter{\hbox{$#2#3$}}\kern-.5\wd0}}

\makeatletter
\@namedef{subjclassname@2020}{\textup{2020} Mathematics Subject Classification}
\makeatother

\begin{document}

\title[Non-harmonic $\ep$-approximability]{$\ep$-Approximability and Quantitative Fatou Property on Lipschitz-graph domains for a class of non-harmonic functions}
\author[T.\ Adamowicz]{Tomasz Adamowicz{\small$^1$}}
\address{The Institute of Mathematics, Polish Academy of Sciences \\ ul. \'Sniadeckich 8, 00-656 Warsaw, Poland}
\email{tadamowi@impan.pl}
\author[M.\ J.\ Gonz\'alez ]{Mar\'ia\ J.\ Gonz\'alez{\small$^2$}}
\address{Departamento de Matem\'aticas, Universidad de C\'adiz, 11510 Puerto Real (C\'adiz), Spain}
\email{majose.gonzalez@uca.es}
\author[M. Grysz\'owka ]{Marcin Grysz\'owka{\small$^1$}}
\address{Institute of Mathematics, Polish Academy of Sciences, ul. \'Sniadeckich 8, 00-656 Warsaw, Poland\/ \and Faculty of Mathematics, Informatics and Mechanics, University of Warsaw, ul. Banacha 2, 02-097 Warsaw, Poland}
\email{mgryszowka@impan.pl}
\thanks{{\small$^1$} T.A. and M.G were supported by the National Science Center, Poland (NCN), UMO-2020/39/O/ST1/00058.
\newline \noindent
{\small$^2$} M.J.G. was  supported in part by the Spanish Ministerio de Ciencia e Innovaci\'on (grant no.  PID2021-123151NB-I00), and  by the grant "Operator Theory: an interdisciplinary approach", reference ProyExcel\_00780, a project financed in the 2021 call for Grants for Excellence Projects, under a competitive bidding regime, aimed at entities qualified as Agents of the Andalusian Knowledge System, in the scope of the Andalusian Research, Development and Innovation Plan (PAIDI 2020). Counseling of University, Research and Innovation of the Junta de Andaluc\'ia.
}
\date{\today}
\begin{abstract}
 We study the class of functions on Lipschitz-graph domains satisfying a differential-oscillation condition and show that such functions are $\ep$-approximable. As a consequence we obtain the quantitative Fatou theorem in the spirit of works e.g. by Garnett~\cite{ga} and Bortz--Hofmann~\cite{bh}. Such a class contains harmonic functions, as well as non-harmonic ones, for example nonnegative subharmonic functions, as illustrated by our discussion.  
\end{abstract}

\maketitle

\section{Introduction and preliminaries}

In this note we consider the Lipschitz-type domains in the form
\begin{equation}\label{def-dom}
 \Om=\{(x,y)\in \R_{+}^{n+1}: y>\phi(x) \},
\end{equation}
where $\phi:\Rn\to \R$ is an $L$-Lipschitz function. On such domains, we would like to study functions $u\in C^2(\Om)$ which satisfy the following condition on any ball $B_r\subset \Om$ such that $2B_r\subset \Om$ :
\begin{equation}\label{cond-star}\tag{$*$}
 \osc_{B_r}(u)\leq C\,\left(r^{1-n}\,\int_{(1+\eta)B_r}(|\nabla u|^2+|u\Delta u|)\,\ud \mathscr{L}^{n+1}\right)^{\frac12}
\end{equation}
for some $\eta\in [0,1)$ and $C>0$. Such a class has been considered in~\cite{g}, when studying the relationships between the nontangential maximal function and convenient versions of the area function of general (nonharmonic) functions. A priori, it might not be clear how wide is this family of functions. However, Proposition 5.1 in~\cite{g} shows that~\eqref{cond-star} follows from the following pointwise condition: 
\begin{equation}\label{cond-main}\tag{$\#$}
|u\Delta u|\le\theta |\nabla u|^2 \textnormal{ in } \Om
\end{equation}
for some $\theta>0$. However, further restriction on $\theta$ is necessary in order to control the area function by the non-tangential function of $u$. Namely, we need to assume that $0<\theta<1$. From now on we will say that a function $u$ satisfies condition~\eqref{cond-main} if $0<\theta<1$. 

The class of functions~\eqref{cond-main} clearly encloses harmonic ones, but also others, see Proposition~\ref{prop-cond-star} in Section 3 below. However, what is perhaps more important from our point of view is that, the oscillation condition~\eqref{cond-star} holds for several non-harmonic examples, for instance for non-negative $C^2$ subharmonic ones, see Proposition~\ref{prop-31} or non-negative $C^2$ functions $u$ with subharmonic $|\nabla u|^\alpha$, for $\alpha\in(0,2]$, see Proposition 3.2 in Section 3. Estimate~\eqref{cond-star} together with~\eqref{cond-main} imply that 
\begin{equation}\label{est-Morrey}
(\osc_{B_r}(u))^2\lesssim_{n,\theta}r^{1-n} \int_{(1+\eta)B_r}|\nabla u|^2\,\ud \mathscr{L}^{n+1},
\end{equation}
which can be understood as the Morrey-type estimate for $u$.

Our main goal is to show the \emph{$\ep$-approximation property} for functions satisfying~\eqref{cond-main} on domains $\Om$. The importance of this condition comes from an observation that a natural candidate for a Carleson measure of a harmonic function, namely $|\nabla u(x)| \ud x$, may fail to be a Carleson measure (see e.g.~\cite[Section 6, Ch. VIII]{ga}). In order to bypass this problem, the notion of $\ep$-approximability has been introduced (\cite{va1, va2}) and has turned out to be important in the studies of the BMO extension problems and Corona theorems (\cite{ga, ht}), the characterization of the uniform rectifiability (\cite{hlm, hmm, hmmtz}) and in the Quantitiative Fatou theorems (\cite{ga, bh}).

\begin{defn}[$\varepsilon$-approximability]\label{def-ea}
Let $\ep>0$ and $\Om\subset \R^{n+1}_{+}$ satisfy~\eqref{def-dom}. We say that a function $u:\Om\rightarrow\mathbb{R}$ is \emph{$\varepsilon$-approximable}, if there exists a function $\varphi\in BV_{loc}(\Om)$ such that 
\begin{enumerate}
\item $\|u-\varphi\|_{\infty}<\varepsilon$,
\item $|\nabla\varphi|\ud y$ defines a Carleson measure on $\Om$, i.e. for every $x\in\partial\Om$
\begin{equation}\label{Carleson-cond}
\sup_{r\in(0,\diam \Om)}\frac{1}{r^n}\int_{\Om\cap B(x,r)}|\nabla u(y)|\ud  \mathscr{L}^{n+1} \leq C_{\varepsilon}.
\end{equation}
\end{enumerate}
\end{defn}
The latter condition can be equivalently formulated in terms of the surface measure, since domain $\Om$ is given by the Lipschitz graph, and thus the surface measure is $n$-Ahlfors regular on the boundary, implying that $\sigma(B(x,r)\cap \partial \Om)\approx r^n$. Let us also add, that regularity of $\ep$-approximation $\phi$ can be improved, see Remark~\ref{eps-reg} at the end of Section 2 below.

Our main goal is to prove the following result.

\begin{theorem}\label{thm-main}
Let $\Om\subset \R^{n+1}_{+}$ be the Lipschitz-graph domain as in~\eqref{def-dom} and let further $u:\Om\rightarrow\mathbb{R}$ be bounded and satisfy condition~\eqref{cond-main}. Then for every $\varepsilon>0$ function $u$ is $\varepsilon$-approximable in $\Om$.
\end{theorem}

The result generalizes the existing ones, as it is to best of our knowledge, first $\ep$-approximability result for functions that need not be solutions of PDEs of divergence form. Moreover, we would like to emphasize that condition~\eqref{cond-main} can be obsolete for some classes of functions and~\eqref{cond-star} instead suffices, as illustrated by nonnegative subharmonic functions, see Proposition~\ref{prop-31} in Section 3. This observation follows from a brief analysis of the proofs of Theorem 1.1 and Lemmas 4.3 and 4.5 in~\cite{g}.

The key consequence of Theorem~\ref{thm-main} is the following Quantitative Fatou Theorem in Corollary~\ref{cor-qft} (see Definition~\ref{def-count-f} of the counting function). 

\begin{cor}[Quantitative Fatou Theorem]\label{cor-qft}
Let $\Om\subset \R^{n+1}_{+}$ be the Lipschitz-graph domain as in~\eqref{def-dom} and let further $u:\Om\rightarrow\mathbb{R}$ satisfy condition~\eqref{cond-main} and be bounded with $\|u\|_{\infty}\le 1$. Then for every point $\om\in\partial\Omega$
\begin{align*}
\sup_{\substack{0<r<r_0}}\frac{1}{r^{n}}\int_{\partial\Omega\cap B(\om,r)}N(r,\ep,\beta)(z)d\sigma(z)\le C(\ep,\alpha,\beta,n,\Omega),
\end{align*}
where $\ep,\alpha,\beta$ are constants in the definition of the counting function $N$. In particular, constant $C$ is a independent of $u$. 
\end{cor}
The proof is a verbatim repetition of the proof of Lemma 2.9 in \cite{kkpt} and, therefore, we omit it.

Let us remark, that the notion of the counting function is known in the literature, see for instance \cite{ga, kkpt, bh}. It provides a way to estimate how much a function oscillates while approaching the boundary. The aforementioned works study the counting function in the context of the Quantitative Fatou Property, which generalizes the classical Fatou theorem stating that non-tangential limit exists at a.e. point of boundary. In the language of counting function it reads that counting function $N$ is finite a.e. 
\medskip

\centerline{ {\sc Preliminaries and notation} }
\medskip

%

Let us recall some necessary definitions and notation used throughout our work. 
\smallskip

\emph{From now on, unless specified otherwise, by $\Om$ we always denote a Lipschitz-type domain as in~\eqref{def-dom}.}

\begin{defn}\label{def-cone} 
For $\alpha>0$ a cone with a vertex at point $(x, \phi(x))\in \partial \Om$ and aperture $\alpha$ is defined as follows
\[
\Gamma_{\alpha}(x):=\{(z,y)\in\mathbb{R}^{n+1}_{+}:|z-x|<\alpha(y-\phi(x))\}.
\]
\end{defn}
Notice that for every $x\in \Rn$ a cone $\Gamma_{\alpha}(x)$ is congruent to a cone $\{(x,y)\in\mathbb{R}^{n+1}_{+}:|x|<\alpha y\}$. However, such cones need not be contained in domain $\Om$. Therefore, we introduce 
the \emph{truncated cone}:
\[
\Gamma_{\alpha,s,t}(x):=\Gamma_{\alpha}(x)\cap\{(z,y):\phi(z)+s<y<\phi(z)+t\},
\]
where $0\le s\le t\le\infty$. In that notation $\Gamma_{\alpha}(x)=\Gamma_{\alpha,0,\infty}(x)$. Since function $\phi$ is $L$-Lipschitz, it holds that $\Gamma_{\alpha,0,t}(x)\subset\Om$ only for $\alpha<\frac{1}{L}$ (and hence, from now on we only consider $\alpha<\frac{1}{L}$). 

%
%

\begin{defn}[Counting function]\label{def-count-f}
Let $\Gamma_{\alpha,0,r}(x)$ be a truncated cone with the vertex at a point $(x, \phi(x)) \in \partial \Om$. Let $u$ be a continuous function defined on $\Omega$. Fix $\ep>0$, $0<\beta<1$ and $0<r<1$. We say that a sequence of points $x_n\in\Gamma_{\alpha, 0, r}(x)$ is $(r,\ep,\beta,x)$-admissible for $u$ if 
\begin{align*}
|u(x_n)-u(x_{n-1})|\ge\ep\quad\hbox{and}\quad |x_n-x|<\beta |x_{n-1}-x|.
\end{align*}
Set 
$$
N(r, \ep, \beta)(x):=\sup\{k: \textnormal{there exists an } (r, \ep, \beta, x)\textnormal{-admissible sequence of length }k\}.
$$
We will call $N$ a \emph{counting function}.
\end{defn}

\begin{defn}[Area function]
Let $f:\Om\rightarrow[0,\infty]$ be a measurable function. The \emph{area function} associated to the density $f$ is defined by
\[
(A_{\alpha}f)(x)=\left(\int_{\Gamma_{\alpha}(x)}f(z,y)(y-\phi(x))^{1-n}\ud z\ud y\right)^{\frac{1}{2}},\quad x\in \Rn.
\]
\end{defn}
Similarly, we define the truncated version of the area function $A_{\alpha,s,t}f$ with respect to cones $\Gamma_{\alpha,s,t}$.

In what follows we are mostly interested in the case $f=|\nabla u|^2$ for a function $u\in C^2(\Om)$. Then we write
\[
(A_{\alpha,s,t}u)(x):=(A_{\alpha,s,t}|\nabla u|^2)(x)=\left(\int_{\Gamma_{\alpha,s,t}(x)}|\nabla u(z,y)|^2(y-\phi(x))^{1-n}\ud z\ud y\right)^{\frac{1}{2}}.
\]
\begin{defn}[Nontangential maximal function]
Let $f:\Om\rightarrow[0,\infty]$ be a continuous function. The \emph{nontangential maximal function function of $u$} is defined as follows
\[
(N_{\alpha}f)(x)=\sup_{\Gamma_{\alpha}(x)}|f(y)|,\quad x\in\Rn. 
\]
\end{defn}
As above, the truncated nontangential maximal function of $u$, denoted by $N_{\alpha,s,t}u$ is defined analogously with respect to cones $\Gamma_{\alpha,s,t}$.

\begin{defn}[Carleson measure]\label{def: Carleson-msp}
 Let $\Omega$ be an open set in $\R^{n+1}$. We say that a (positive) Borel measure $\mu$ on $\Omega$ is an \emph{$\alpha$-Carleson measure on  $\Omega$}, if there exists a constant $C>0$ such that
\begin{equation*}
\mu(\Omega \cap B(x,r))\leq C r^{n},\quad \text{for
all }x \in \partial\Omega\text{ and }r>0.
\end{equation*}
The \emph{Carleson measure constant} of $\mu$ is defined as the infimum of constants $C$ above.
\end{defn}

\begin{defn}[Local BV functions]
 Let $\Omega$ be an open set in $\R^{n+1}$. We say that an $L^1_{loc}$-function $f$ has \emph{locally bounded variation in $\Om$}, and denote it by $f\in BV_{loc}(\Om)$, if for any open set $\Om'\Subset \Om$ the total variation of $f$ over $\Om'$ is finite:
\begin{equation*}
 \sup_{\tiny{\Psi\in C_{0}^{1}(\Om', \R^{n+1}),\,\|\Psi\|_{L^{\infty}}\leq 1}}\int_{\Om'} f(x)\,{\rm div } \Psi (x)\,\ud x <\infty.
\end{equation*}
\end{defn}
Recall, that the latter expression defines a (Radon) measure on $\Om$, see \cite[Section 5.1]{eg}.

\section{Proof of Theorem~\ref{thm-main}}

Let us briefly describe our approach to the proof of the main result.  First, we recall and formulate some auxiliary notions regarding curved cubes, the associated red and blue sets and the stopping condition allowing us to choose the appropriate families of cubes. Then we construct function $\phi_1$, the first approximation of $\phi$, see~\eqref{def-phi1} and show in Proposition~\ref{Claim1} that $\phi_1$ is the Carleson measure. The proof of Proposition~\ref{Claim1} relies on two auxiliary observations, namely Lemmas~\ref{lem1} and~\ref{lem2}. The first one gives a lower bound estimate for area function and is applied in the proof of Lemma~\ref{lem2} to control the sum of volumes of cubes obtained by the stopping procedure. Then, we construct the function $\phi$, see~\eqref{def-eps-apr} and show that it $\ep$-approximates function $u$ in the $L^\infty$-norm. In order to show condition (2) in Definition~\ref{def-ea}, we study the decomposition of the gradient of $\phi$, see~\eqref{Carl-decomp}, and show that each of its terms leads to the Carleson condition, see estimates (Car1) and (Car2). An important auxiliary result, perhaps of the independent interest, is presented in Proposition~\ref{prop24} and proved in the Appendix. It gives the $L^2$ bounds for the area function on cubes. The above approach has been inspired by the discussion in~\cite[Section 6, Ch. VIII]{ga}) and also by~\cite{hmm}.

Let us first set up the stage for a geometric constructions we use to prove our result.
\smallskip

{\bf Curved cubes and associated centers}. Fix $\varepsilon>0$ and denote by $Q_0$ the unit cube in $\Rn$. We denote by $\{Q_{j_1,\ldots, j_n}^m\}$ the family of dyadic cubes in the dyadic decomposition of $Q_0$:
$$
Q_{j_1,\ldots, j_n}^m=\{(x_1,\ldots,x_n) \in\mathbb{R}^n: j_i 2^{-m}\le x_i\le (j_i+1)2^{-m}\},\quad \hbox{for } m\in\mathbb{N} \hbox{ and } j_1,\dots,j_n\in \{0,\ldots, 2^m-1\}. 
$$
In the case parameters $m$ and $j_1,\ldots,j_n$ are fixed or their exact values are not important for the discussion, we will write $Q$ to denote a cube in the $m$-th generation for some $m$. For the sake of notation, in what follows we will usually denote the side length of $Q$ by $l(Q)$ rather than $2^{-m}$.

Let further
$$
\hat{Q}_0=\{(x,y)\in\mathbb{R}^{n+1}:x\in Q_0, \phi(x)\le y\le 1+\phi(x)\}
$$
be an associated curved unit cube in $\mathbb{R}^{n+1}$, where $\phi :Q_0\rightarrow\mathbb{R}$ is a Lipschitz function. Similarly, for a given cube $Q$, we define the \emph{curved cube}
\[
\hat{Q}=\{(x,y)\in\mathbb{R}^{n+1}:x\in Q, \phi(x)\le y\le \phi(x)+l(Q)\}.
\]
In what follows we will often omit the word \emph{curved} when discussing sets $\hat{Q}$ and instead simply write cube.

Let $x_{\hat{Q}}$ denote a center of a (curved) cube $\hat{Q}$, i.e. $x_{\hat{Q}}:=(x_{Q},\phi(x_Q)+2^{-m-1})$, where $x_Q$ is a center of $Q$. Note that since by~\eqref{def-dom} it holds that $\phi$ is $L$-Lipschitz, we have the following inclusions:
\begin{equation}\label{ball-box}
B\big(x_{\hat{Q}}, l(Q)\big) \subset \hat{Q}\subset B\big(x_{\hat{Q}}, C(n,L)l(Q)\big),\quad C(n,L):= \sqrt{n+1+L^2}.
\end{equation}
Moreover, we define the \emph{associated center of $\hat{Q}$} as follows:
\begin{equation}\label{def-ass-center}
x^{l}_{\hat{Q}}=x_{\hat{Q}}+\overline{e_{n+1}}l(Q)=(x_Q,\phi(x_Q)+2^{-m-1}+l(Q))=\left(x_Q,\phi(x_Q)+\frac{3}{2}2^{-m}\right).
\end{equation}
The name of this point is justified by the fact that $x^{l}_{\hat{Q}}$ does not lie inside $\hat{Q}$, and is the center of the  curved cube lying directly above cube $\hat{Q}$ and obtained by shifting up $\hat{Q}$ in $l(Q)$, see Figure 1.
\begin{figure}[h]
  \centering
  \begin{minipage}[b]{0.4\textwidth}
    \definecolor{uuuuuu}{rgb}{0.26666666666666666,0.26666666666666666,0.26666666666666666}
\begin{tikzpicture}[scale=0.22][line cap=round,line join=round,>=triangle 45,x=1cm,y=1cm]
\begin{axis}[
x=1cm,y=1cm,
axis x line=middle,
axis y line=none,
xmin=-3.334991070439598,
xmax=15.037116070994667,
ymin=0,
ymax=25.97316353722639,
xtick=\empty,
ytick={-10,0,...,40},]
\clip(-0.334991070439598,-1.436300260977678) rectangle (14.037116070994667,30.97316353722639);
\draw [line width=2pt] (0,4)-- (2,2);
\draw [line width=2pt] (6,5)-- (2,2);
\draw [line width=2pt] (6,5)-- (7,3);
\draw [line width=2pt] (7,3)-- (8,4);
\draw [line width=2pt] (8,4)-- (10,2);
\draw [line width=2pt] (0,4)-- (0,14);
\draw [line width=2pt] (0,14)-- (2,12);
\draw [line width=2pt] (2,12)-- (6,15);
\draw [line width=2pt] (6,15)-- (7,13);
\draw [line width=2pt] (7,13)-- (8,14);
\draw [line width=2pt] (8,14)-- (10,12);
\draw [line width=2pt] (10,12)-- (10,2);
\draw [line width=2pt,dash pattern=on 1pt off 2pt] (0,14)-- (0,24);
\draw [line width=2pt,dash pattern=on 1pt off 2pt] (0,24)-- (2,22);
\draw [line width=2pt,dash pattern=on 1pt off 2pt] (2,22)-- (6,25);
\draw [line width=2pt,dash pattern=on 1pt off 2pt] (6,25)-- (7,23);
\draw [line width=2pt,dash pattern=on 1pt off 2pt] (7,23)-- (8,24);
\draw [line width=2pt,dash pattern=on 1pt off 2pt] (8,24)-- (10,22);
\draw [line width=2pt,dash pattern=on 1pt off 2pt] (10,22)-- (10,12);
\draw (5.046398583363668,10.912990401724846) node[scale=4][anchor=north west] {$x_{\hat{Q}}$};
\draw (5.046398583363668,1.6715182553375463) node[scale=4][anchor=north west] {$x_Q$};
\draw (5.046398583363668,20.865345020911167) node[scale=4][anchor=north west] {$x_{\hat{Q}}^{l}$};
\draw (12.014299592262855,2.007274951677194) node[scale=3][anchor=north west] {$\mathbb{R}^n$};
\begin{scriptsize}
\draw [fill=uuuuuu] (5.000613617654646,9.250460213240984) circle (2pt);
\draw [fill=black] (5,0) circle (2.5pt);
\draw [fill=uuuuuu] (4.998249011993611,19.24868675899521) circle (2pt);
\end{scriptsize}
\end{axis}
\end{tikzpicture}
    \caption{\small The center $x_{\hat{Q}}$ of a cube $\hat{Q}$ and its associated center $x_{\hat{Q}}^{l}$.}
 \end{minipage}
  \begin{minipage}[b]{0.4\textwidth}
    \definecolor{ffttww}{rgb}{1,0.2,0.4}
\definecolor{zzttqq}{rgb}{0.6,0.2,0}
\definecolor{uuuuuu}{rgb}{0.26666666666666666,0.26666666666666666,0.26666666666666666}
\begin{tikzpicture}[scale=0.22][line cap=round,line join=round,>=triangle 45,x=1cm,y=1cm]
\begin{axis}[
x=1cm,y=1cm,
axis x line=middle,
axis y line=none,
xmin=-3.508318523625785,
xmax=13.490725158501984,
ymin=0,
ymax=21.671283001188026,
xtick=\empty,
ytick={0,5,...,25},]
\clip(-0.508318523625785,-0.264312172935558) rectangle (15.490725158501984,25.671283001188026);
\fill[line width=2pt,dash pattern=on 1pt off 1pt,color=zzttqq,fill=zzttqq,fill opacity=0.10000000149011612] (0,19) -- (2,17) -- (6,20) -- (7,18) -- (8,19) -- (10,17) -- (10,7) -- (8,9) -- (7,8) -- (6,10) -- (2,7) -- (0,9) -- cycle;
\draw [line width=2pt] (0,4)-- (2,2);
\draw [line width=2pt] (6,5)-- (2,2);
\draw [line width=2pt] (6,5)-- (7,3);
\draw [line width=2pt] (7,3)-- (8,4);
\draw [line width=2pt] (8,4)-- (10,2);
\draw [line width=2pt] (0,4)-- (0,14);
\draw [line width=2pt] (0,14)-- (2,12);
\draw [line width=2pt] (2,12)-- (6,15);
\draw [line width=2pt] (6,15)-- (7,13);
\draw [line width=2pt] (7,13)-- (8,14);
\draw [line width=2pt] (8,14)-- (10,12);
\draw [line width=2pt] (10,12)-- (10,2);
\draw (5.0229275714402455,20.29170765734653) node[scale=3][anchor=north west] {$x_{\hat{Q}}^{l}$};
\draw (10.908815888819746,1.697331826900235) node[scale=3][anchor=north west] {$\mathbb{R}^n$};
\draw [line width=2pt,dash pattern=on 1pt off 1pt] (0,19)-- (0,14);
\draw [line width=2pt,dash pattern=on 1pt off 1pt,color=zzttqq] (0,19)-- (2,17);
\draw [line width=2pt,dash pattern=on 1pt off 1pt,color=zzttqq] (2,17)-- (6,20);
\draw [line width=2pt,dash pattern=on 1pt off 1pt,color=zzttqq] (6,20)-- (7,18);
\draw [line width=2pt,dash pattern=on 1pt off 1pt,color=zzttqq] (7,18)-- (8,19);
\draw [line width=2pt,dash pattern=on 1pt off 1pt,color=zzttqq] (8,19)-- (10,17);
\draw [line width=2pt,dash pattern=on 1pt off 1pt,color=zzttqq] (10,17)-- (10,7);
\draw [line width=2pt,dash pattern=on 1pt off 1pt,color=zzttqq] (10,7)-- (8,9);
\draw [line width=2pt,dash pattern=on 1pt off 1pt] (8,9)-- (7,8);
\draw [line width=2pt,dash pattern=on 1pt off 1pt,color=zzttqq] (7,8)-- (6,10);
\draw [line width=2pt,dash pattern=on 1pt off 1pt,color=zzttqq] (6,10)-- (2,7);
\draw [line width=2pt,dash pattern=on 1pt off 1pt,color=zzttqq] (2,7)-- (0,9);
\draw [line width=2pt,dash pattern=on 1pt off 1pt,color=zzttqq] (0,9)-- (0,19);
\draw (1.542025878366348,6.811124736896671) node[scale=3][anchor=north west] {$\hat{Q}$};
\draw [color=ffttww](1.0357129048283267,17.063962451041636) node[scale=3][anchor=north west] {$T(\hat{Q})$};
\begin{scriptsize}
\draw [fill=uuuuuu] (4.998249011993611,19.24868675899521) circle (2pt);
\end{scriptsize}
\end{axis}
\end{tikzpicture}

    \caption{\small The set $T(\hat{Q})$ (brown set) with respect to the set $\hat{Q}$ (set bounded by black line).}
  \end{minipage}
\end{figure}

\begin{figure}
    \centering
    \definecolor{zzttqq}{rgb}{0.6,0.2,0}
\begin{tikzpicture}[scale=0.4][line cap=round,line join=round,>=triangle 45,x=1cm,y=1cm]
\clip(-0.2503456737040986,0.6736986668985544) rectangle (13.196190342174265,10.866408183307874);
\fill[line width=2pt,color=zzttqq,fill=zzttqq,fill opacity=0.10000000149011612] (0,1) -- (0,5) -- (2,6) -- (4,5.6) -- (4,1.6) -- (2,2) -- cycle;
\fill[line width=2pt,color=zzttqq,fill=zzttqq,fill opacity=0.10000000149011612] (5,1.4) -- (5,2.4) -- (5.999605663713984,2.200000155501131) -- (5.999211389500886,1.200157722099823) -- cycle;
\fill[line width=2pt,color=zzttqq,fill=zzttqq,fill opacity=0.10000000149011612] (6,3.2) -- (7,3) -- (8,4) -- (8,2) -- (7,1) -- (5.999211389500886,1.200157722099823) -- cycle;
\draw [line width=1pt] (0,1)-- (2,2);
\draw [line width=1pt] (2,2)-- (7,1);
\draw [line width=1pt] (7,1)-- (8,2);
\draw [line width=1pt] (0,1)-- (0,9);
\draw [line width=1pt] (0,9)-- (2,10);
\draw [line width=1pt] (2,10)-- (7,9);
\draw [line width=1pt] (7,9)-- (8,10);
\draw [line width=1pt] (8,10)-- (8,2);
\draw [line width=1pt] (0,5)-- (2,6);
\draw [line width=1pt] (2,6)-- (4,5.6);
\draw [line width=1pt] (4,5.6)-- (4,1.6);
\draw [line width=1pt] (8,4)-- (7,3);
\draw [line width=1pt] (7,3)-- (6,3.2);
\draw [line width=1pt] (6,3.2)-- (5.999211389500886,1.200157722099823);
\draw [line width=1pt] (5.999605663713984,2.200000155501131)-- (5,2.4);
\draw [line width=1pt] (5,2.4)-- (5,1.4);
\draw [line width=1pt,color=zzttqq] (0,1)-- (0,5);
\draw [line width=1pt,color=zzttqq] (0,5)-- (2,6);
\draw [line width=1pt,color=zzttqq] (2,6)-- (4,5.6);
\draw [line width=1pt,color=zzttqq] (4,5.6)-- (4,1.6);
\draw [line width=1pt,color=zzttqq] (4,1.6)-- (2,2);
\draw [line width=1pt,color=zzttqq] (2,2)-- (0,1);
\draw [line width=1pt,color=zzttqq] (5,1.4)-- (5,2.4);
\draw [line width=1pt,color=zzttqq] (5,2.4)-- (5.999605663713984,2.200000155501131);
\draw [line width=1pt,color=zzttqq] (5.999605663713984,2.200000155501131)-- (5.999211389500886,1.200157722099823);
\draw [line width=1pt,color=zzttqq] (5.999211389500886,1.200157722099823)-- (5,1.4);
\draw [line width=1pt,color=zzttqq] (6,3.2)-- (7,3);
\draw [line width=1pt,color=zzttqq] (7,3)-- (8,4);
\draw [line width=1pt,color=zzttqq] (8,4)-- (8,2);
\draw [line width=1pt,color=zzttqq] (8,2)-- (7,1);
\draw [line width=1pt,color=zzttqq] (7,1)-- (5.999211389500886,1.200157722099823);
\draw [line width=1pt,color=zzttqq] (5.999211389500886,1.200157722099823)-- (6,3.2);
\draw (1.608993500668486,4.495964735552049) node[scale=0.7][anchor=north west] {$\hat{Q}_1$};
\draw (6.447341436599947,2.692639205725785) node[scale=0.7][anchor=north west] {$\hat{Q}_2$};
\draw (4.9,2.4182201033609187) node[scale=0.7][anchor=north west] {$\hat{Q}_3$};
\draw (3.37590675040304,8.573048542115776) node[anchor=north west] {$R(\hat{Q})$};
\end{tikzpicture}

    \caption{\small An example of how a set $R(\hat{Q})$ may look like. Cubes $\hat{Q}_1, \hat{Q}_2, \hat{Q}_3$ are removed from cube $\hat{Q}$ to obtain $R(\hat{Q})$. In general there may be infinitely many sets that are removed from $\hat{Q}$.}
    
\end{figure}
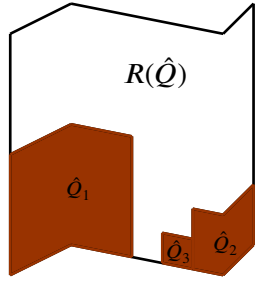

\smallskip

{\bf Stopping time conditions}. Let us now define a stopping procedure.

Set 
\begin{align*}
G_0&:=\{\hat{Q_0}\} \hbox{ and denote by } \\
G_1&:=\hbox{ a family of maximal curved cubes $\hat{Q}\subset\hat{Q_0}$ such that }
|u(x_{\hat{Q_0}})-u(x^{l}_{\hat{Q}})|>\varepsilon.
\end{align*}
Next, define
$$
G_2=\bigcup_{\hat{Q}\in G_1}G_1(\hat{Q}),
$$
where $G_1(\hat{Q})$ is defined the same way as $G_1$ with $\hat{Q_0}$ replaced with $\hat{Q}$.
Then define inductively families of sets $G_k$, for $k=2,\ldots$. Denote by $G=\bigcup_{k=0}G_k$.
Let us introduce a domain which roughly can be understood as follows: given any curved cube $\hat{Q}\in G$ consider its subset constructed by removing those maximal curved cubes $\hat{Q}_i$, where the jump of the values of $u$ at associated centers is big: $|u(x_{\hat{Q}})-u(x^{l}_{\hat{Q}_i})|>\varepsilon$, i.e. we define
\[
R(\hat{Q}):=\hat{Q}\setminus\bigcup_{\hat{Q}_i\in G_1(\hat{Q})}\hat{Q}_i,\quad \hbox{ for any }\hat{Q}\in G.
\]
Thus, set $R(\hat{Q})$ consists of all curved subcubes in $\hat{Q}$ with small oscillations of $u$, see Figure 3. We remark that this construction is similar to the one in Garnett's book, see the proof of Theorem 6.1 in~\cite[Section 6 in Ch. VIII]{ga}. 

Notice that given two different sets $\widehat{Q}, \widehat{W}\in G$, the corresponding domains $R(\widehat{Q})$ and $R(\widehat{W})$ can only intersect piecewise along boundaries, but their interiors are pairwise disjoint. Finally, we define blue and red sets, which are essential in our construction. Denote by $T(\hat{Q})$ the set $\hat{Q}$ translated vertically by $\frac{1}{2}l(Q)$:
\begin{equation}\label{def-TQ}
T(\hat{Q})=\{(x,y)\in\mathbb{R}^{n+1}:x\in Q,\phi(x)+\frac{1}{2}l(Q)\le y\le \phi(x)+\frac{3}{2}l(Q)\}.
\end{equation}
The key feature of sets $T(\hat{Q})$, to which we appeal several times below, is that they are separated from the graph of the Lipschitz function $\phi$, i.e. from the boundary of $\Om$. Moreover, an important feature of sets $T(\hat{Q})$ is that the associated center of $\hat{Q}$ is the center of an upper side of $T(\hat{Q})$, see Figure 2.

Sets $T(\hat{Q})$ are not disjoint. However, for a given set $\hat{Q}$ a set $T(\hat{Q})$ intersects only finitely many other sets of form $T(\hat{Q}_j)$. Moreover, the cardinality of a family of sets $\# \{j: T(\hat{Q})\cap T(\hat{Q}_j)\neq\emptyset\}$ is uniformly bounded for all choices of $\hat{Q}$.
When dealing with set $\hat{Q}_0$, we set $T(\hat{Q}_0):=\{(x,y)\in\mathbb{R}^{n+1}:x\in Q_0,\phi(x)+\frac{1}{2}l(Q_0)\le y\le \phi(x)+l(Q_0)\}$, i.e. its upper half.

Let $k>0$. We say that $T(\hat{Q})$ is \emph{blue}, if
\[
\osc_{T(\hat{Q})}u\le k\varepsilon.
\]
Otherwise, we say that $T(\hat{Q})$ is \emph{red}. 

\begin{proof}[Proof of Theorem~\ref{thm-main}]
First, we define an auxiliary function $\phi_1: \bigcup_{\hat{Q}_k\in G} R(\hat{Q}_k)\to \R$, which later on will be used to define the $\ep$-approximation of $u$, cf.~\eqref{def-eps-apr} 
\begin{equation}\label{def-phi1}
\varphi_1(z):=\sum_{j=1}^{\infty}\sum_{\hat{Q}_k\in G_j} u(x_{\hat{Q}_k})\chi_{R(\hat{Q}_k)}(z).
\end{equation}

Notice that, $\phi_1$ is in fact defined for all $z\in \hat{Q}$ and, moreover, for any $\hat{Q} \in G$ it holds that
\begin{equation}\label{ineq-phi1-Carl}
\int_{\hat{Q}}|\nabla\varphi_1| \,\ud \mathscr{L}^{n+1}\le\sum_{\hat{Q}_j\in G} |\hat{Q}\cap\partial R(\hat{Q}_j)|,
\end{equation}
where $|\cdot|$ denotes the $n$-Hausdorff measure. Here, the expression $|\nabla\varphi_1|$ is understood only in the distributional sense and the component functions of $\nabla\varphi_1$ are the signed measures supported on the appropriate faces in $\partial R(\hat{Q}_j)$, see the discussion for the upper-half space in $\R^2$ on pg. 345 in~\cite[Section 6, Ch. VIII]{ga}. Therefore, $|\chi_{R(\hat{Q}_k)}|$ in~\eqref{def-phi1} are the $n$-Hausdorff measures of $\hat{Q}\cap\partial R(\hat{Q}_j)$ and the above estimate is justified.

Our first step is to prove the following observation, which applied at~\eqref{ineq-phi1-Carl} shows that $|\nabla\varphi_1| \,\ud \mathscr{L}^{n+1}$ is a Carleson measure.

\begin{prop}\label{Claim1} For any $\hat{Q}$ it holds that
$\sum_{\hat{Q}_j\in G}|\hat{Q}\cap\partial R(\hat{Q}_j)|\le C\varepsilon^{-2}l(Q)^n$.
\end{prop}
\begin{proof} We may assume, without loss of the generality, that $\hat{Q}\in G$. For otherwise, we consider a family $M(\hat{Q})$ of cubes such that $\hat{Q}_1\in M(\hat{Q})$ if $\hat{Q}_1\subset\hat{Q}$, $\hat{Q}_1\in G$ and $\hat{Q}_1$ is maximal. Then it suffices to prove the assertion for each of the cubes in $M(\hat{Q})$. Hence, from now on we assume that $\hat{Q}\in G$. In order to show the assertion of Proposition~\ref{Claim1} we consider two cases depending whether $\hat{Q}_j$ is contained in $\hat{Q}$ or not and then prove two auxiliary observations in Lemmas~\ref{lem1} and~\ref{lem2}.  
\smallskip

\noindent {\sc Case 1:} $\hat{Q}_j$ is such that $\hat{Q}\cap\partial R(\hat{Q}_j)\neq\emptyset$ and $\hat{Q}_j\not\subset\hat{Q}$. 
	\begin{enumerate}
	\item[(1.1)] Let $l(Q_j)\le l(Q)$. Then, it holds that ${\rm int }\hat{Q}\cap {\rm int }\hat{Q}_j=\emptyset$, but the boundaries of curved cubes $\hat{Q}$ and $\hat{Q}_j$ still intersect.

It holds that $\hat{Q}\cap\partial R(\hat{Q}_j)$ is a subset of the vertical faces of $\hat{Q}$ (throughout the paper, by vertical faces we mean those different from the bottom and the top deck of a cube/curved cube). It is the case, since: (1) $\hat{Q}_j$ has to touch $\hat{Q}$, as otherwise $\hat{Q}\cap\partial R(\hat{Q}_j)=\emptyset$ and such a curved cube does not contribute to the sum $\sum_{\hat{Q}_j\in G}|\hat{Q}\cap\partial R(\hat{Q}_j)|$; (2) since $l(Q_j)\le l(Q)$, only vertical sides can touch.

For different curved cubes $\hat{Q}_j$ satisfying $l(Q_j)\le l(Q)$, the corresponding sets $\hat{Q}\cap\partial R(\hat{Q}_j)$ can intersect along a set of positive $(n-1)$-Hausdorff measure only, due to the definition of $G$ and $R(\hat{Q}_j)$. Indeed, let $\hat{Q}_l\not=\hat{Q}_k$ be such cubes. Then we have three cases: 

\noindent
(a) cubes $\hat{Q}_l$ and $\hat{Q}_k$ have no common face and ${\rm int }\hat{Q}_l \cap {\rm int }\hat{Q}_k=\emptyset$ in which case the corresponding sets $\hat{Q}\cap\partial R(\hat{Q}_k)$ and $\hat{Q}\cap\partial R(\hat{Q}_l)$ can intersect along a set of positive $(n-1)$-Hausdorff measure only. See Figure 4.

\noindent
(b) cubes $\hat{Q}_l$ and $\hat{Q}_k$ have a common face and ${\rm int }\hat{Q}_l \cap {\rm int }\hat{Q}_k=\emptyset$. Then sets $\hat{Q}\cap\partial R(\hat{Q}_k)$ and $\hat{Q}\cap\partial R(\hat{Q}_l)$ are subsets of a common face of $\hat{Q}$, which can only intersect along an $(n-1)$ dimensional set $\partial\hat{Q}\cap\partial\hat{Q}_k\cap\partial\hat{Q}_l$.

\noindent
(c) interiors of cubes $\hat{Q}_j$ and $\hat{Q}_k$ intersect, but this means that one of the cubes contains another, for instance let $\hat{Q}_j\subset \hat{Q}_k$. However, then $\hat{Q}_j\cap R(\hat{Q}_k)=\emptyset$ and so the conclusion is as in case (a) above.

 Therefore, all such $\hat{Q}_j$ amount to at most $C(n)l(Q)^n$ in $\sum_{\hat{Q}_j\in G}|\hat{Q}\cap\partial R(\hat{Q}_j)|$, as they cover at most all vertical faces of $\hat{Q}$.  

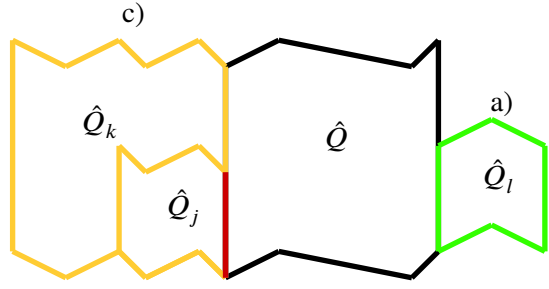
\begin{figure}
    \centering
  \definecolor{ccqqqq}{rgb}{0.8,0,0}
\definecolor{qqqqff}{rgb}{0,0,1}
\definecolor{ffcctt}{rgb}{1,0.8,0.2}
\definecolor{ttffqq}{rgb}{0.2,1,0}

\begin{tikzpicture}[scale=0.35][line cap=round,line join=round,>=triangle 45,x=1cm,y=1cm]
\clip(-9.032531984347367,-0.0269002934497455) rectangle (12.456791632011594,11.41300342670833);
\draw [line width=2pt] (0,1)-- (2,2);
\draw [line width=2pt] (2,2)-- (7,1);
\draw [line width=2pt] (7,1)-- (8,2);
\draw [line width=2pt] (0,1)-- (0,9);
\draw [line width=2pt] (0,9)-- (2,10);
\draw [line width=2pt] (2,10)-- (7,9);
\draw [line width=2pt] (7,9)-- (8,10);
\draw [line width=2pt] (8,10)-- (8,2);
\draw [line width=2pt,color=ttffqq] (8,2)-- (10,3);
\draw [line width=2pt,color=ttffqq] (10,3)-- (12,2);
\draw [line width=2pt,color=ttffqq] (12,2)-- (12,6);
\draw [line width=2pt,color=ttffqq] (12,6)-- (10,7);
\draw [line width=2pt,color=ttffqq] (10,7)-- (8,6);
\draw [line width=2pt,color=ttffqq] (8,6)-- (8,2);
\draw [line width=2pt,color=ffcctt] (0,1)-- (-1,2);
\draw [line width=2pt,color=ffcctt] (-1,2)-- (-3,1);
\draw [line width=2pt,color=ffcctt] (-3,1)-- (-4,2);
\draw [line width=2pt,color=ffcctt] (-4,2)-- (-6,1);
\draw [line width=2pt,color=ffcctt] (-6,1)-- (-8,2);
\draw [line width=2pt,color=ffcctt] (-8,2)-- (-8,10);
\draw [line width=2pt,color=ffcctt] (-8,10)-- (-6,9);
\draw [line width=2pt,color=ffcctt] (-6,9)-- (-4,10);
\draw [line width=2pt,color=ffcctt] (-4,10)-- (-3,9);
\draw [line width=2pt,color=ffcctt] (-3,9)-- (-1,10);
\draw [line width=2pt,color=ffcctt] (-1,10)-- (0,9);
\draw [line width=2pt,color=ffcctt] (0,9)-- (0,1);
\draw [line width=2pt,color=qqqqff] (0,5)-- (0,1);
\draw [line width=2pt,color=ffcctt] (-4,6)-- (-3,5);
\draw [line width=2pt,color=ffcctt] (-3,5)-- (-1,6);
\draw [line width=2pt,color=ffcctt] (-1,6)-- (0,5);
\draw [line width=2pt,color=ffcctt] (0,5)-- (0,1);
\draw [line width=2pt,color=ffcctt] (0,1)-- (-1,2);
\draw [line width=2pt,color=ffcctt] (-1,2)-- (-3,1);
\draw [line width=2pt,color=ffcctt] (-3,1)-- (-4,2);
\draw [line width=2pt,color=ffcctt] (-4,2)-- (-4,6);
\draw [line width=2pt,color=ccqqqq] (0,5)-- (0,1);
\draw (3.3718743307866665,7.323503923905612) node[anchor=north west] {$\hat{Q}$};
\draw (-5.75672045746338,7.934988742257292) node[anchor=north west] {$\hat{Q}_k$};
\draw (-2.568263904629631,4.746532189423532) node[anchor=north west] {$\hat{Q}_j$};
\draw (9.312012566202965,5.882146852076652) node[anchor=north west] {$\hat{Q}_l$};
\draw (-4.271685898609305,11.82228508749297) node[anchor=north west] {c)};
\draw (9.574077488353684,8.459118586558732) node[anchor=north west] {a)};
\end{tikzpicture}
    \caption{\small This figure illustrates a) and c) in Case 1.1 in Proposition \ref{Claim1}. Since b) may only be observed if the dimension is greater than two, it is not shown as a figure. Red line is a set $\partial\hat{Q}\cap\partial\hat{Q}_j$. A set $\partial R(\hat{Q}_j)\cap\hat{Q}$ is a subset of a red set, whereas a set $\partial R(\hat{Q}_k)\cap\hat{Q}$ is contained in a yellow line above a red one. Therefore these sets may only intersect along a set of dimension $n-1$.}
    \label{fig:enter-label}
\end{figure}

\item[(1.2)] Let $l(Q_j)>l(Q)$.

Then, there are at most $C(n)$ of such cubes $\hat{Q}_j$. In order to see that this holds, let us consider two cases. If $\hat{Q}\not\subset\hat{Q}_j$, then there cannot be more of such $\hat{Q}_j$ than faces of $\hat{Q}$. This is a consequence of the following observations: (1) $\hat{Q}\cap\partial R(\hat{Q}_j)\neq\emptyset$ by assumptions, and so $\hat{Q}$ and $\hat{Q}_j$ have to touch; (2) since $\hat{Q}_j\in G$ and $l(Q_j)>l(Q)$, then for each face $F$ of $\hat{Q}$ there is at most one curved cube in $G$ such that it touches $F$ with the face of side length bigger than $l(Q)$ and, moreover, $\hat{Q}\cap\partial R(\hat{Q}_j)\neq\emptyset$ (see also Figure 5).

Let now $\hat{Q}\subset\hat{Q}_j$, then there exists exactly one cube in family $G$ such that   $\hat{Q}\cap\partial R(\hat{Q}_j)\neq\emptyset$. To prove it, note that for any bigger cube $\hat{Q}_k\in G$ with $\hat{Q}\subset\hat{Q}_j\subset\hat{Q}_k$ it holds that $\hat{Q}\cap\partial R(\hat{Q}_k)=\emptyset$, as for such $\hat{Q}_k$, the cube $\hat{Q}_j$ is not contained in $R(\hat{Q}_k)$, as it had to be removed in the construction of $R(\hat{Q}_k)$. Therefore, there is only one cube such that $\hat{Q}\subset\hat{Q}_j$ and $\hat{Q}\cap\partial R(\hat{Q}_j)\neq\emptyset$.

Thus, similarly to case (1.1), such cubes contribute at most $C(n)l(Q)^n$ to the sum $\sum_{\hat{Q}_j\in G}|\hat{Q}\cap\partial R(\hat{Q}_j)|$.  
\end{enumerate}
In summary, the discussion in cases (1.1) and (1.2) gives that
 \begin{equation}\label{claim-case1}
 \sum_{\hat{Q}_j\in G, \hat{Q}_j\not \subset\hat{Q}}|\hat{Q}\cap\partial R(\hat{Q}_j)|\le C(n)l(Q)^n.
 \end{equation}

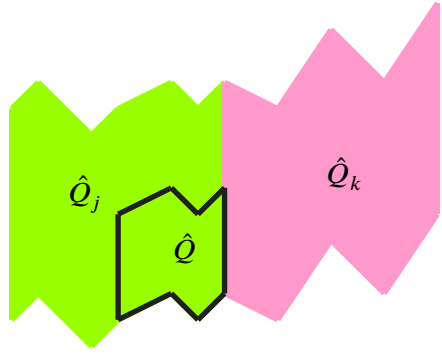
\begin{figure}
    \centering
    \definecolor{sqsqsq}{rgb}{0.12549019607843137,0.12549019607843137,0.12549019607843137}
\definecolor{ffzzcc}{rgb}{1,0.6,0.8}
\definecolor{zzffqq}{rgb}{0.6,1,0}
\begin{tikzpicture}[scale=0.35][line cap=round,line join=round,>=triangle 45,x=1cm,y=1cm]
\clip(-3.5254794412759756,-0.4768734734871145) rectangle (18.516666293012918,13.386198196913542);
\fill[line width=2pt,color=zzffqq,fill=zzffqq,fill opacity=0.1] (0,1) -- (1,2) -- (3,0) -- (4,1) -- (6,2) -- (7,1) -- (8,2) -- (8,10) -- (7,9) -- (6,10) -- (4,9) -- (3,8) -- (1,10) -- (0,9) -- cycle;
\fill[line width=2pt,color=ffzzcc,fill=ffzzcc,fill opacity=0.1] (8,2) -- (10,1) -- (12,4) -- (14,2) -- (16,5) -- (16,13) -- (14,10) -- (12,12) -- (10,9) -- (8,10) -- cycle;
\draw [line width=2pt,color=zzffqq] (0,1)-- (1,2);
\draw [line width=2pt,color=zzffqq] (1,2)-- (3,0);
\draw [line width=2pt,color=zzffqq] (3,0)-- (4,1);
\draw [line width=2pt,color=zzffqq] (4,1)-- (6,2);
\draw [line width=2pt,color=zzffqq] (6,2)-- (7,1);
\draw [line width=2pt,color=zzffqq] (7,1)-- (8,2);
\draw [line width=2pt,color=zzffqq] (8,2)-- (8,10);
\draw [line width=2pt,color=zzffqq] (8,10)-- (7,9);
\draw [line width=2pt,color=zzffqq] (7,9)-- (6,10);
\draw [line width=2pt,color=zzffqq] (6,10)-- (4,9);
\draw [line width=2pt,color=zzffqq] (4,9)-- (3,8);
\draw [line width=2pt,color=zzffqq] (3,8)-- (1,10);
\draw [line width=2pt,color=zzffqq] (1,10)-- (0,9);
\draw [line width=2pt,color=zzffqq] (0,9)-- (0,1);
\draw [line width=2pt,color=ffzzcc] (8,2)-- (10,1);
\draw [line width=2pt,color=ffzzcc] (10,1)-- (12,4);
\draw [line width=2pt,color=ffzzcc] (12,4)-- (14,2);
\draw [line width=2pt,color=ffzzcc] (14,2)-- (16,5);
\draw [line width=2pt,color=ffzzcc] (16,5)-- (16,13);
\draw [line width=2pt,color=ffzzcc] (16,13)-- (14,10);
\draw [line width=2pt,color=ffzzcc] (14,10)-- (12,12);
\draw [line width=2pt,color=ffzzcc] (12,12)-- (10,9);
\draw [line width=2pt,color=ffzzcc] (10,9)-- (8,10);
\draw [line width=2pt,color=ffzzcc] (8,10)-- (8,2);
\draw [line width=2pt,color=sqsqsq] (4,1)-- (6,2);
\draw [line width=2pt,color=sqsqsq] (6,2)-- (7,1);
\draw [line width=2pt,color=sqsqsq] (7,1)-- (8,2);
\draw [line width=2pt,color=sqsqsq] (8,2)-- (8,6);
\draw [line width=2pt,color=sqsqsq] (8,6)-- (7,5);
\draw [line width=2pt,color=sqsqsq] (7,5)-- (6,6);
\draw [line width=2pt,color=sqsqsq] (6,6)-- (4,5);
\draw [line width=2pt,color=sqsqsq] (4,5)-- (4,1);
\draw (1.7594419894475388,6.867739311865788) node[anchor=north west] {$\hat{Q}_j$};
\draw (11.460054352420306,7.563123710645196) node[anchor=north west] {$\hat{Q}_k$};
\draw (5.688363842551204,4.677278455710651) node[anchor=north west] {$\hat{Q}$};
\end{tikzpicture}
    \caption{\small This figure shows Case 1.2 in Proposition \ref{Claim1}. The purple cube refers to the case $\hat{Q}\not\subset\hat{Q}_j$ and the green one refers to the case $\hat{Q}\subset\hat{Q}_j$.}
    \label{fig:enter-label}
\end{figure}

\noindent {\sc Case 2:} $\hat{Q}_j$ is such that $\hat{Q}\cap\partial R(\hat{Q}_j)\neq\emptyset$ and $\hat{Q}_j\subset\hat{Q}$.

Then, trivially we have that 
\begin{equation}\label{claim-case2}
\sum_{\hat{Q}_j\in G, \hat{Q}_j\subset\hat{Q}}|\hat{Q}\cap\partial R(\hat{Q}_j)|\le C(n)\sum_{\hat{Q}_j\in G,\hat{Q}_j\subset\hat{Q}}l(Q_j)^n.
\end{equation}

To continue the proof of Proposition \ref{Claim1} let us prove the following observation.
\begin{lem}\label{lem1}
Let $\hat{Q}\in G$. It holds that 
\[
\sum_{\hat{Q}_j\in G_1(\hat{Q})}l(Q_j)^n\le C\varepsilon^{-2}\int_{\widetilde{R(\hat{Q})}}|\nabla u(x,y)|^2(y-\phi(x))\,\ud x \ud y,  
\]
where $C=C(n, L, \theta, \eta)$ and the set $\widetilde{R(\hat{Q})}$ is defined as follows:
 \begin{equation}\label{def-R-tilde}
 \widetilde{R(\hat{Q})}:=\bigcup_{\hat{Q}_j\in G_1(\hat{Q})}\widetilde{\hat{Q_j}}\quad\hbox{ where }\quad
 \widetilde{\hat{Q_j}}:=
 \begin{cases}
 T(\hat{Q}_j),\quad \hbox{ if } T(\hat{Q}_j) \hbox{ is red }\\
 \bigcup_{X\in U\hat{Q}_j}\Gamma_{\alpha,0,\frac{1}{2}l(Q_j)}(X),\quad \hbox{ if }T(\hat{Q}_j) \hbox{ is blue}.
 \end{cases}
 \end{equation}
 By $U\hat{Q}_j$, we denote the upper deck of $\hat{Q}_j$. (We refer to the discussion in the proof below, see~\eqref {def-R-tilde2}, where the set $\widetilde{R(\hat{Q})}$ is constructed and its meaning explained).
\end{lem}
\begin{proof}
Let $\hat{Q}_j\in G_1(\hat{Q})$. 
\smallskip

\noindent {\sc Case 1:} The translated curved cube $T(\hat{Q}_j)$ is red (cf.~\eqref{def-TQ} for the definition of  $T(\hat{Q}_j)$). 
Then, it follows by~\eqref{est-Morrey} and~\eqref{ball-box} that
\begin{align*}
k^2\varepsilon^2\leq (\osc_{T(\hat{Q}_j)} u)^2\lesssim_{n, L, \theta, \eta} l(Q_j)^{1-n}\int_{T(\hat{Q}_j)}|\nabla u|^2,
\end{align*}
for some $k>0$ whose exact value will be determined later in this proof. Hence, since $T(\hat{Q}_j)\cap \partial \Om=\emptyset$ we have that $y-\phi(x)\approx_{n, L} l(Q_j)$ for all $(x,y)\in T(\hat{Q}_j)$. Thus, we get
\begin{equation}\label{est-case1}
k^2l(Q_j)^n\lesssim_{n, L, \theta, \eta}\varepsilon^{-2}\int_{T(\hat{Q}_j)}|\nabla u(x,y)|^2(y-\phi(x))\,\ud x \ud y.
\end{equation}
\noindent {\sc Case 2:} Set $T(\hat{Q}_j)$ is blue.

Since $\hat{Q}_j\in G_1(\hat{Q})$, we know that $|u(x^{l}_{\hat{Q}_j})-u(x_{\hat{Q}})|>\varepsilon$. Next, let us define the point 
\[
x^{\frac12 l}_{\hat{Q}_j}:= x^{\hat{Q}_j}+\frac12l(Q_j)\overline{e_{n+1}},
\]
 which has the same $x$ coordinate as the center of the curved cube $x_{\hat{Q}_j}$ but its $y$ coordinate equals $\phi(x)+l(Q_j)$. Thus, one can think that such point is a vertical projection of the center of the cube $\hat{Q}_j$ on the upper deck of $\hat{Q}_j$, denoted by $U\hat{Q}_j$. However, notice that $x^{\frac12 l}_{\hat{Q}_j}$ does not lie in the boundary $\partial \Om$ while we would like to consider a cone with the vertex at that point. Therefore, we let $\Omega_j=\Omega+\overline{e_{n+1}}l(Q_j)$ be a subdomain of $\Om$ obtained by shifting $\Om$ vertically up by $l(Q_j)$. Now $x^{\frac12 l}_{\hat{Q}_j} \in\partial\Omega_j$.

Therefore, we have 
\begin{equation}\label{est-nontm}
N_{\alpha,0,\frac{1}{2}l(Q_j)}(u-u(x_{\hat{Q}}))(x^{\frac12 l}_{\hat{Q}_j})>\varepsilon,
\end{equation}
 where the (truncated) non-tangential maximal function $N$ is considered with respect to domain $\Om_j$.

 We now show that estimate~\eqref{est-nontm} holds not only at $x^{\frac12 l}_{\hat{Q}_j}$, the center of the upper deck of $\hat{Q}_j$, but in fact  at its all points $X$, i.e. 
\[
 N_{\alpha,0,\frac{1}{2}l(Q_j)}(u-u(P_{\hat{Q}}))(X)\gtrsim\epsilon. 
\]
Let us consider vertical shifts of points $X\in U\hat{Q}_j$ so that they belong to $T(\hat{Q}_j)\setminus \hat{Q}_j$, e.g. $X+\tfrac14\overline{e_{n+1}}l(Q_j)$ and notice that they satisfy
\[
X+\tfrac14\overline{e_{n+1}}l(Q_j) \in\Gamma_{\alpha}(X) \quad\hbox{ and }\quad X+\tfrac14\overline{e_{n+1}}l(Q_j) \in T(\hat{Q}_j).
\]
As a consequence we get, by the triangle inequality and since $T(\hat{Q}_j)$ is blue, that
\begin{align*}
\varepsilon&<|u(x^{l}_{\hat{Q}})-u(x_{\hat{Q}})|\le |u(x^{l}_{\hat{Q}})-u(X+\tfrac14\overline{e_{n+1}}l(Q_j))|+|u(X+\tfrac14\overline{e_{n+1}}l(Q_j))-u(x_{\hat{Q}})|\\
&\le k\varepsilon+|u(X+\tfrac14\overline{e_{n+1}}l(Q_j))-u(x_{\hat{Q}})|
\end{align*}
and hence
\begin{equation}\label{aux3-lem2}
|u(X+\tfrac14\overline{e_{n+1}}l(Q_j))-u(x_{\hat{Q}})|>(1-k)\varepsilon.
\end{equation}
Therefore, for every $X\in U\hat{Q}_j$, we obtain the following estimate
\[
(1-k)\varepsilon \leq N_{\alpha,0,\frac{1}{2}l(Q_j)}(u-u(x_{\hat{Q}}))(X).
\]
Hence, for any $X\in U\hat{Q}_j$
\begin{align}
(1-k)^2\varepsilon^2 l(Q_j)^n&\lesssim_{n, L} (N_{\alpha,0,\frac{1}{2}l(Q_j)}(u-u(x_{\hat{Q}})))^2(X) \int_{U\hat{Q}_j} \ud \mathscr{H}^n \nonumber\\
&\le \int_{U\hat{Q}_j} (N_{\alpha,0,\frac{1}{2}l(Q_j)}(u-u(x_{\hat{Q}})))^2(X)\, \ud \mathscr{H}^n\nonumber\\
&\lesssim \int_{U\hat{Q}_j} \int_{\Gamma_{\alpha,0, 1/2 l(Q_j)}(X)} |\nabla u|^2(y-\phi(x)-l(Q_j))^{1-n}\,\ud x\ud y \tag{$N\lesssim A$}\nonumber\\
&\lesssim \int_{\bigcup \limits_{X\in U\hat{Q}_j}\Gamma_{\alpha,0, 1/2l(Q_j)}(X)}|\nabla u|^2(y-\phi(x)-l(Q_j))\,\ud x\ud y \tag{Fubini's Theorem}\nonumber\\
&\le \int_{\bigcup \limits_{X\in U\hat{Q}_j}\Gamma_{\alpha,0, 1/2l(Q_j)}(X)}|\nabla u|^2(y-\phi(x))\,\ud x\ud y, 
\label{est-case2}
\end{align}
where the third ($N\lesssim A$) inequality follows by the fact that $U\hat{Q}_j\subset \partial \Omega_j$ and by applying the local version of Theorem 1.1 (b) for $p=2$ in~\cite{g} allowing us the consider the truncated versions of the $N_{\alpha}$ and the $A_{\alpha}$ functions, see the comment following the statement of Theorem 1.2 in~\cite{g}.
  
 Notice that the set $\bigcup_{X\in U\hat{Q}_j}\Gamma_{\alpha,0,\frac{1}{2}l(Q_j)}(X)$ consists of the upper-half of $T(\hat{Q}_j)$ and additional parts belonging to neighbouring curved cubes.
 
\begin{figure}[!tbp]
  \centering
  \begin{minipage}[b]{0.45\textwidth}
    \definecolor{qqqqff}{rgb}{0,0,1}
\definecolor{ffqqqq}{rgb}{1,0,0}
\begin{tikzpicture}[scale=0.6][line cap=round,line join=round,>=triangle 45,x=1cm,y=1cm]
\clip(-0.0783342931457541,-0.020217484616638) rectangle (14.770348486221739,11.014420765662203); rectangle (14.770348486221739,11.014420765662203);
\fill[line width=1pt,dotted,color=ffqqqq,fill=ffqqqq,fill opacity=0.12] (1,3) -- (3,4) -- (5,3) -- (5,7) -- (3,8) -- (1,7) -- cycle;
\fill[line width=1pt,dotted,color=qqqqff,fill=qqqqff,fill opacity=0.24] (8,5) -- (10,6) -- (12,5) -- (12,7) -- (10,8) -- (8,7) -- cycle;
\fill[line width=1pt,dotted,color=qqqqff,fill=qqqqff,fill opacity=0.21] (12,7) -- (12.786532022402879,6.6833207515257245) -- (12,5) -- cycle;
\draw [line width=1pt] (1,1)-- (3,2);
\draw [line width=1pt] (3,2)-- (5,1);
\draw [line width=1pt] (5,1)-- (5,5);
\draw [line width=1pt] (5,5)-- (3,6);
\draw [line width=1pt] (3,6)-- (1,5);
\draw [line width=1pt] (1,5)-- (1,1);
\draw [line width=1pt,dotted,color=ffqqqq] (1,3)-- (3,4);
\draw [line width=1pt,dotted,color=ffqqqq] (3,4)-- (5,3);
\draw [line width=1pt,dotted,color=ffqqqq] (5,3)-- (5,7);
\draw [line width=1pt,dotted,color=ffqqqq] (5,7)-- (3,8);
\draw [line width=1pt,dotted,color=ffqqqq] (3,8)-- (1,7);
\draw [line width=1pt,dotted,color=ffqqqq] (1,7)-- (1,3);
\draw [line width=1pt] (8,1)-- (10,2);
\draw [line width=1pt] (10,2)-- (12,1);
\draw [line width=1pt] (12,1)-- (12,5);
\draw [line width=1pt] (12,5)-- (10,6);
\draw [line width=1pt] (10,6)-- (8,5);
\draw [line width=1pt] (8,5)-- (8,1);
\draw [shift={(8,5)},line width=1pt,dotted,color=qqqqff,fill=qqqqff,fill opacity=0.22]  (0,0) --  plot[domain=1.5707963267948966:2.007128639793479,variable=\t]({1*2*cos(\t r)+0*2*sin(\t r)},{0*2*cos(\t r)+1*2*sin(\t r)}) -- cycle ;
\draw [line width=1pt,dotted,color=qqqqff] (8,5)-- (10,6);
\draw [line width=1pt,dotted,color=qqqqff] (10,6)-- (12,5);
\draw [line width=1pt,dotted,color=qqqqff] (12,7)-- (10,8);
\draw [line width=1pt,dotted,color=qqqqff] (10,8)-- (8,7);
\draw (2.5806882094621693,7.592020479986081) node[anchor=north west] {$\widetilde{\hat{Q}_k}$};
\draw (9.594496202329028,7.739901973811592) node[anchor=north west] {$\widetilde{\hat{Q}_j}$};
\draw [line width=1pt] (8,1)-- (7,1.5);
\draw [line width=1pt] (11.999086218827804,1.000456890586098)-- (13,0.6);
\draw [line width=1pt,dotted,color=qqqqff] (12,7)-- (12.786532022402879,6.6833207515257245);
\draw [line width=1pt,dotted,color=qqqqff] (12.786532022402879,6.6833207515257245)-- (12,5);
\draw [line width=1pt,dotted,color=qqqqff] (12,5)-- (12,7);
\end{tikzpicture}
    \caption{\small This figure shows how sets $\widetilde{\hat{Q}_k}$ and $\widetilde{\hat{Q}_j}$ look like for $T(\hat{Q}_k)$ red and $T(\hat{Q}_j)$ blue, respectively. Notice that for a blue $T(\hat{Q_j})$ we drew a bit more of a graph of $\phi$ as a blue set is a union of truncated cones and the way in which the cone is truncated depends on $\phi$.}
  \end{minipage}
  \hfill
  \begin{minipage}[b]{0.4\textwidth}
    \definecolor{qqqqff}{rgb}{0,0,1}
\definecolor{ffqqqq}{rgb}{1,0,0}
\begin{tikzpicture}[scale=0.3][line cap=round,line join=round,>=triangle 45,x=1cm,y=1cm]
\clip(-1.7595675301932747,0.33314910893249794) rectangle (20.170896014719542,21.5679332966804);
\fill[line width=1pt,color=ffqqqq,fill=ffqqqq,fill opacity=0.54] (1,5) -- (5,7) -- (6,6) -- (9,8) -- (9,16) -- (6,14) -- (5,15) -- (1,13) -- cycle;
\fill[line width=1pt,color=qqqqff,fill=qqqqff,fill opacity=0.5] (9,10) -- (9,8) -- (10,7) -- (13,8) -- (13,10) -- (10,9) -- cycle;
\fill[line width=1pt,color=qqqqff,fill=qqqqff,fill opacity=0.54] (13,6) -- (15,4) -- (15,5) -- (13,7) -- cycle;
\fill[line width=1pt,color=ffqqqq,fill=ffqqqq,fill opacity=0.56] (17,2) -- (16,2) -- (15,3) -- (15,5) -- (16,4) -- (17,4) -- cycle;
\fill[line width=1pt,color=qqqqff,fill=qqqqff,fill opacity=0.51] (13,8) -- (13,10) -- (13.636029765969887,9.363970234030113) -- cycle;
\draw [line width=1pt] (1,1)-- (5,3);
\draw [line width=1pt] (5,3)-- (6,2);
\draw [line width=1pt] (6,2)-- (9,4);
\draw [line width=1pt] (9,4)-- (10,3);
\draw [line width=1pt] (10,3)-- (13,4);
\draw [line width=1pt] (13,4)-- (16,1);
\draw [line width=1pt] (16,1)-- (17,1);
\draw [line width=1pt] (17,1)-- (17,17);
\draw [line width=1pt] (17,17)-- (16,17);
\draw [line width=1pt] (16,17)-- (13,20);
\draw [line width=1pt] (13,20)-- (10,19);
\draw [line width=1pt] (10,19)-- (9,20);
\draw [line width=1pt] (9,20)-- (6,18);
\draw [line width=1pt] (6,18)-- (5,19);
\draw [line width=1pt] (5,19)-- (1,17);
\draw [line width=1pt] (1,17)-- (1,1);
\draw [line width=0.3pt,color=ffqqqq] (1,5)-- (5,7);
\draw [line width=0.3pt,color=ffqqqq] (5,7)-- (6,6);
\draw [line width=0.3pt,color=ffqqqq] (6,6)-- (9,8);
\draw [line width=0.3pt,color=ffqqqq] (9,8)-- (9,16);
\draw [line width=0.3pt,color=ffqqqq] (9,16)-- (6,14);
\draw [line width=0.3pt,color=ffqqqq] (6,14)-- (5,15);
\draw [line width=0.3pt,color=ffqqqq] (5,15)-- (1,13);
\draw [line width=0.3pt,color=ffqqqq] (1,13)-- (1,5);
\draw [line width=0.3pt,color=qqqqff] (9,10)-- (9,8);
\draw [line width=0.3pt,color=qqqqff] (9,8)-- (10,7);
\draw [line width=0.3pt,color=qqqqff] (10,7)-- (13,8);
\draw [line width=0.3pt,color=qqqqff] (13,8)-- (13,10);
\draw [line width=0.3pt,color=qqqqff] (13,10)-- (10,9);
\draw [line width=0.3pt,color=qqqqff] (10,9)-- (9,10);
\draw [shift={(13,6)},line width=0.3pt,color=qqqqff,fill=qqqqff,fill opacity=0.5]  (0,0) --  plot[domain=1.5707963267948966:2.0071286397934798,variable=\t]({1*1*cos(\t r)+0*1*sin(\t r)},{0*1*cos(\t r)+1*1*sin(\t r)}) -- cycle ;
\draw [line width=0.3pt,color=qqqqff] (13,6)-- (15,4);
\draw [line width=0.3pt,color=qqqqff] (15,4)-- (15,5);
\draw [line width=0.3pt,color=qqqqff] (15,5)-- (13,7);
\draw [line width=0.3pt,color=qqqqff] (13,7)-- (13,6);
\draw [line width=0.3pt,color=ffqqqq] (17,2)-- (16,2);
\draw [line width=0.3pt,color=ffqqqq] (16,2)-- (15,3);
\draw [line width=0.3pt,color=ffqqqq] (15,3)-- (15,5);
\draw [line width=0.3pt,color=ffqqqq] (15,5)-- (16,4);
\draw [line width=0.3pt,color=ffqqqq] (16,4)-- (17,4);
\draw [line width=0.3pt,color=ffqqqq] (17,4)-- (17,2);
\draw (11.614989035141308,16.47793199806682) node[anchor=north west] {$\hat{Q}$};
\draw [line width=0.3pt,color=qqqqff] (13,8)-- (13,10);
\draw [line width=0.3pt,color=qqqqff] (13,10)-- (13.636029765969887,9.363970234030113);
\draw [line width=0.3pt,color=qqqqff] (13.636029765969887,9.363970234030113)-- (13,8);
\end{tikzpicture}

    \caption{\small This figure shows how a domain $\widetilde{R(\hat{Q})}$ is constructed. It is a union of red and blue sets of the form $\widetilde{\hat{Q}_k}$.}
  \end{minipage}
\end{figure}

 However, those parts may only be contained in cubes in the same generation (in the dyadic decomposition) as $T(\hat{Q}_j)$ and intersect only finitely many of such cubes whose number is estimated by a constant $C(n,\alpha)$. To be more specific, notice that the distance of a point in $\Gamma_{\alpha,0,\frac12l(Q_j)}(X)$ to the axis of the cone can be at most $\frac12\alpha l(Q_j)$. Hence, as we are only interested in cubes in the same generation as $T(\hat{Q}_j)$, in each direction such a cone can only intersect at most $\lceil\tfrac{\alpha}{2}\rceil$ other cubes. Moreover, as faces of $\hat{Q}_j$ are $n$-dimensional, a cone can overlap with up to $\om_n(\lceil\frac{\alpha}{2}\rceil)^n$ other cubes, where $\om_n$ stands for the measure of $n$-dimensional unit ball. 
  Therefore, upon adding up in~\eqref{est-case2} over all cubes $\hat{Q}_j\in G_1(\hat{Q})$, we increase the constant on the right-hand side only by a factor of $C(n,\alpha)+1$. Thus, also the discussion of case 2 is completed.
 
 In order to estimate the sum in the assertion of the lemma we now combine cases 1 and 2. For this, we also need to analyze how a red set $T(\hat{Q}_j)$ may intersect other red sets. Notice that the case of cubes  in the same generation as a red $T(\hat{Q}_j)$ is already taken care of above. However, it may happen that $T(\hat{Q}_j)$ intersects with sets that belong to one generation below the one of $T(\hat{Q}_j)$, i.e. to $(m-1)$-th generation for $T(\hat{Q}_j)$ belonging to the $m$-th generation, for some $m$. However, since the number of such cubes is finite,  $T(\hat{Q}_j)$ can only intersect $C(n)$ of such cubes.  
  
 Finally, we combine estimates~\eqref{est-case1} and~\eqref{est-case2} to arrive at the assertion of Lemma~\ref{lem1}:
\[
\sum_{\hat{Q}_j\in G_1(\hat{Q})}l(Q_j)^n\le C\varepsilon^{-2}\int_{\widetilde{R(\hat{Q})}}|\nabla u(y)|^2(y-\phi(x)),
\]
where $\widetilde{R(\hat{Q})}:=\bigcup_{\hat{Q}_j\in G_1(\hat{Q})}\widetilde{\hat{Q_j}}$ with
\begin{equation}\label{def-R-tilde2}
 \widetilde{\hat{Q_j}}=
 \begin{cases}
 T(\hat{Q}_j),\quad \hbox{ if } T(\hat{Q}_j) \hbox{ is red }\\
 \bigcup_{X\in U\hat{Q}_j}\Gamma_{\alpha,0,\frac{1}{2}l(Q_j)}(X),\quad \hbox{ if }T(\hat{Q}_j) \hbox{ is blue}.
 \end{cases}
 \end{equation}
See Figures 6 and 7 illustrating the construction of the set $\widetilde{R(\hat{Q})}$.

Notice, that by~\eqref{est-case1} and~\eqref{est-case2}, the assertion of the lemma holds with $C$ depending on $\max\{k^{-2}, (1-k)^{-2}\}$ and, thus taking into account also~\eqref{aux3-lem2}, any $0<k<1$ is suitable.
\end{proof}

Lemma~\ref{lem1} implies the following observation. 
\begin{lem}\label{lem2}
Let $\hat{Q}\in G$. Then,
$\sum_{\hat{Q}_j\in G,\hat{Q}_j\subset\hat{Q}}l(Q_j)^n\le C\varepsilon^{-2}l(Q)^n$.
\end{lem}
\begin{proof}
It holds that
\begin{align*}
\sum_{\hat{Q}_j\in G,\hat{Q}_j\subset\hat{Q}}l(Q_j)^n&=\sum_{k\ge 0}\sum_{\hat{Q}_j\in G_k(\hat{Q})}l(Q_j)^n\\
&=l(Q)^n+\sum_{k\ge 1}\sum_{\hat{Q}_j\in G_k(\hat{Q})}l(Q_j)^n\\
&=l(Q)^n+\sum_{k\ge 1}\,\,\sum_{\hat{Q}^{'}\in G_{k-1}(\hat{Q})}\,\,\sum_{\hat{Q}_j\in G_1(\hat{Q}^{'})}l(Q_j)^n\\
&\lesssim_{n, L, \theta, \eta} l(Q)^n+\varepsilon^{-2}\sum_{k\ge 1}\sum_{\hat{Q}^{'}\in G_{k-1}(\hat{Q})}\int_{\widetilde{R(\hat{Q}^{'})}}|\nabla u(x,y)|^2(y-\phi(x))\,\ud x \ud y\quad \tag{Lemma~\ref{lem1}}\\
&\lesssim_{n, L, \theta, \eta} l(Q)^n+\varepsilon^{-2}\int_{C(n,\alpha)\hat{Q}}|\nabla u(x,y)|^2(y-\phi(x))\,\ud x \ud y,\end{align*}
where the second inequality follows, by the discussion similar to the one at the end of the proof of Lemma~\ref{lem1}, from the fact that any cube may be counted at most finitely many times with the uniform constant depending on $n$ and $\alpha$. However, since sets $\widetilde{R(\hat{Q}^{'})}$ may contain also unions of cones, we may need to consider a cube bigger than $\hat{Q}$ so that $\bigcup \widetilde{R(\hat{Q}^{'})}\subset C(n,\alpha)\hat{Q}$. The proof of Lemma~\ref{lem2} will be completed once we show that
\begin{equation}\label{lem3-final-est}
\int_{C(n,\alpha)\hat{Q}}|\nabla u(x,y)|^2(y-\phi(x))\,\ud x \ud y \lesssim_{n, L, \theta, \eta} l(Q)^n.
\end{equation}
In order to prove this estimate, observe that for any point $\om\in C(n,\alpha)\hat{Q}\cap \partial \Om$, i.e. in the bottom face of cube $C(n,\alpha)\hat{Q}$, it holds that
\[
 z\in \Gamma_{\alpha, 0, C(n,\alpha)l(Q)}(\om)\,\Leftrightarrow\, \om \in B(z, (1+\alpha)d(z, \partial \Om)) \cap \partial \Om.
\]
The set on the right-hand side is sometimes called a shadow of point $z$. Moreover, notice that for $z=(x,y)\in C(n,\alpha)\hat{Q}$ it holds that $y-\phi(x) \lesssim_{n,L} d(z, \partial \Om)$. These observations together with the Fubini theorem allow us to obtain the following estimate
\begin{align}
&\int_{C(n,\alpha)\hat{Q}}|\nabla u(x,y)|^2(y-\phi(x))\,\ud x \ud y \nonumber \\
&\approx \int_{C(n,\alpha)\hat{Q}}|\nabla u(x,y)|^2 (y-\phi(x))^{1-n} (y-\phi(x))^{n}\,\ud x \ud y \label{lem2-est333}\\
&\approx_{n,L, \alpha}\int_{C(n,\alpha)\hat{Q}}|\nabla u(x,y)|^2 (y-\phi(x))^{1-n} \left(\int_{\partial \Om} \chi_{B(z, (1+\alpha)d(z, \partial \Om)) \cap \partial \Om}\ud \sigma \right)\,\ud x \ud y \nonumber  \\
&\approx_{n,L, \alpha}\int_{C(n,\alpha)\hat{Q}}|\nabla u(x,y)|^2 (y-\phi(x))^{1-n} \left(\int_{\partial \Om} \chi_{\Gamma_{\alpha, 0, l(Q)}}\ud \sigma \right) \,\ud x \ud y \nonumber \\
&\approx_{n,L, \alpha}\int_{\partial \Om} \left(\int_{C(n,\alpha)\hat{Q}}|\nabla u(x,y)|^2 (y-\phi(x))^{1-n}  \chi_{\Gamma_{\alpha, 0, l(Q)}} \,\ud x \ud y\right)\, \ud \sigma\tag{\small{Fubini's theorem}}\nonumber \\
&\approx_{n,L, \alpha}\int_{Q} \left(A_{\alpha, 0, C(n,\alpha)l(Q)} u\right)^2(x)\,\ud x
\lesssim (C(n,\alpha)l(Q))^n. \nonumber
\end{align}
The last inequality follows from the following observation, whose proof we present in the appendix.
\begin{prop}\label{prop24}
 Let $\Om\subset \R^{n+1}_{+}$ be the Lipschitz-graph domain as in~\eqref{def-dom} and let further $u:\Om\rightarrow\mathbb{R}$ be bounded and satisfy condition~\eqref{cond-main}. Then for any dyadic cube $Q\subset \Rn$ it holds that 
 \[
  \int_{Q} \left(A_{\alpha, 0, l(Q)} u\right)^2(x)\,\ud x<c (l(Q))^n,
 \]
 where the constant $c$ depends only on $\alpha$, $\theta$ as in~\eqref{cond-main}, $n$ and the Lipschitz constant $L$ of $\phi$.
\end{prop}
Therefore, the inequality~\eqref{lem3-final-est} is proven and, hence the proof of Lemma~\ref{lem2} is completed.
\end{proof}
Upon combining the discussion in~\eqref{claim-case1} and~\eqref{claim-case2} together with Lemma~\ref{lem1} we complete the proof of Proposition~\ref{Claim1}. 
\end{proof} 
{\sc Continuation of the proof of Theorem~\ref{thm-main}}:

Recall, that as already mentioned in the discussion following~\eqref{def-phi1} and~\eqref{ineq-phi1-Carl}, Proposition~\ref{Claim1} shows that $|\nabla\varphi_1| \ud \mathscr{L}^{n+1}$ is a Carleson measure in $\Om$. Let us now define the following function $\varphi:\Om\to\R$:
\begin{equation}\label{def-eps-apr}
\varphi(z)=
\begin{cases}
u(z) &\textnormal{ if $z$ belongs to any red } T(\hat{Q})\\
\varphi_1(z) & \textnormal{ otherwise}
\end{cases}
\end{equation}
Our goal is to prove that $\varphi$ is an $\varepsilon$-approximation of $u$ as in Definition~\ref{def-ea}. Denote by
\smallskip

\centerline{
 {\sc Red} the union of all red sets $T(\hat{Q})$ and by {\sc Blue} the union of all blue sets $T(\hat{Q})$,  for $\hat{Q}\subset \hat{Q_0}$.
 }
\smallskip

If $z\in\hbox{\sc Red}$, then $u(z)-\varphi(z)=0$, whereas if $z\in\hbox{\sc Blue}$, then $z\in R(\hat{Q})$ for some $\hat{Q}\in G$. Suppose that $z\in T(\hat{Q}_1)$. Since, by the definition~\eqref{def-TQ}, the set $T(\hat{Q}_1)$ is a vertical translation of cube $\hat{Q}_1$, its upper half may be a subset of one $R$-set (i.e. $R(\hat{Q})$ for some $\hat{Q}\in G$), while its lower half may lie in another $R$-set.  Moreover, it can also happen that $T(\hat{Q}_1)$ is entirely contained in one $R$-set. This discussion leads to the following two cases. 

If $T(\hat{Q}_1)\subset R(\hat{Q})$, then
\begin{align*}
|u(z)-\varphi(z)|&=|u(z)-\varphi_1(z)| \\
&=|u(z)-u(x_{\hat{Q}_1})\chi_{R(\hat{Q})}(z)| \le |u(z)-u(x^{l}_{\hat{Q}_1})|+ |u(x^{l}_{\hat{Q}_1})-u(x_{\hat{Q}})|\le k\varepsilon + \varepsilon,
\end{align*}
where the first $k\varepsilon$ comes from the fact that $T(\hat{Q}_1)$ is blue and the second $\varepsilon$ is obtained because $\hat{Q}_1$ is not entirely removed from $R(\hat{Q})$.

If $T(\hat{Q}_1)\not\subset R(\hat{Q})$, then suppose first that $z$ belongs to the lower half of $T(\hat{Q}_1)$, i.e. $z\in T(\hat{Q}_1)\cap \hat{Q}_1$. Then $z\in R(\hat{Q}_1)$ and since $T(\hat{Q}_1)$ is blue, we have that $|u(z)-\varphi(z)|\le k\varepsilon$.

If $z$ belongs to the upper half of $T(\hat{Q}_1)$, i.e. $z\in T(\hat{Q}_1)\setminus \hat{Q}_1$, then $z$ lies in $T(P\hat{Q}_1)$, where $P\hat{Q}_1$ denotes the \emph{parent} of $\hat{Q}_1$, meaning the smallest cube $\hat{Q}$ containing $\hat{Q}_1$. Moreover, $z\in T(P\hat{Q}_1)\cap P\hat{Q}_1$. If $T(P\hat{Q}_1)$ is blue, then $|u(z)-\varphi(z)|\le k\varepsilon$ in a similar manner as before. Otherwise, if $T(P\hat{Q}_1)$ is red, then this case has already been taken care of above (see also Figure 8).

\begin{figure}
    \centering
    \definecolor{ududff}{rgb}{0.30196078431372547,0.30196078431372547,1}
\definecolor{ffqqtt}{rgb}{1,0,0.2}
\definecolor{qqzzqq}{rgb}{0,0.6,0}
\definecolor{ffwwzz}{rgb}{1,0.4,0.6}
\definecolor{ffxfqq}{rgb}{0.8,0.4980392156862745,0}
\definecolor{zzffzz}{rgb}{0.6,1,0.6}
\definecolor{zzttqq}{rgb}{0.6,0.2,0}
\begin{tikzpicture}[scale=0.55][line cap=round,line join=round,>=triangle 45,x=1cm,y=1cm]
\clip(-4.317607402969818,-0.8004920755514504) rectangle (13.866858647739768,16.198030537068384);

\fill[line width=2pt,color=zzffzz,fill=zzffzz,pattern=north east lines,pattern color=zzffzz] (1,1) -- (2,2) -- (3,0) -- (4,1) -- (5,3) -- (5,7) -- (4,5) -- (3,4) -- (2,6) -- (1,5) -- cycle;
\fill[line width=2pt,color=ffxfqq,fill=ffxfqq,fill opacity=0.17] (1,5) -- (2,6) -- (3,4) -- (4,5) -- (5,7) -- (7,5) -- (8,6) -- (9,6) -- (9,14) -- (8,14) -- (7,13) -- (5,15) -- (4,13) -- (3,12) -- (2,14) -- (1,13) -- cycle;
\fill[line width=2pt,color=ffwwzz,fill=ffwwzz,fill opacity=0.73] (1,7) -- (1,3) -- (2,4) -- (3,2) -- (4,3) -- (5,5) -- (5,9) -- (4,7) -- (3,6) -- (2,8) -- cycle;
\draw [line width=2pt,color=zzttqq] (1,1)-- (2,2);
\draw [line width=2pt,color=zzttqq] (2,2)-- (3,0);
\draw [line width=2pt,color=zzttqq] (3,0)-- (4,1);
\draw [line width=2pt,color=zzttqq] (4,1)-- (5,3);
\draw [line width=2pt,color=zzttqq] (5,3)-- (7,1);
\draw [line width=2pt,color=zzttqq] (7,1)-- (8,2);
\draw [line width=2pt,color=zzttqq] (8,2)-- (9,2);
\draw [line width=2pt,color=zzttqq] (9,2)-- (9,10);
\draw [line width=2pt,color=zzttqq] (9,10)-- (8,10);
\draw [line width=2pt,color=zzttqq] (8,10)-- (7,9);
\draw [line width=2pt,color=zzttqq] (7,9)-- (5,11);
\draw [line width=2pt,color=zzttqq] (5,11)-- (4,9);
\draw [line width=2pt,color=zzttqq] (4,9)-- (3,8);
\draw [line width=2pt,color=zzttqq] (3,8)-- (2,10);
\draw [line width=2pt,color=zzttqq] (2,10)-- (1,9);
\draw [line width=2pt,color=zzttqq] (1,9)-- (1,1);
\draw [line width=2pt,color=zzffzz] (1,1)-- (2,2);
\draw [line width=2pt,color=zzffzz] (2,2)-- (3,0);
\draw [line width=2pt,color=zzffzz] (3,0)-- (4,1);
\draw [line width=2pt,color=zzffzz] (4,1)-- (5,3);
\draw [line width=2pt,color=zzffzz] (5,3)-- (5,7);
\draw [line width=2pt,color=zzffzz] (5,7)-- (4,5);
\draw [line width=2pt,color=zzffzz] (4,5)-- (3,4);
\draw [line width=2pt,color=zzffzz] (3,4)-- (2,6);
\draw [line width=2pt,color=zzffzz] (2,6)-- (1,5);
\draw [line width=2pt,color=zzffzz] (1,5)-- (1,1);
\draw [line width=2pt,color=ffxfqq] (1,5)-- (2,6);
\draw [line width=2pt,color=ffxfqq] (2,6)-- (3,4);
\draw [line width=2pt,color=ffxfqq] (3,4)-- (4,5);
\draw [line width=2pt,color=ffxfqq] (4,5)-- (5,7);
\draw [line width=2pt,color=ffxfqq] (5,7)-- (7,5);
\draw [line width=2pt,color=ffxfqq] (7,5)-- (8,6);
\draw [line width=2pt,color=ffxfqq] (8,6)-- (9,6);
\draw [line width=2pt,color=ffxfqq] (9,6)-- (9,14);
\draw [line width=2pt,color=ffxfqq] (9,14)-- (8,14);
\draw [line width=2pt,color=ffxfqq] (8,14)-- (7,13);
\draw [line width=2pt,color=ffxfqq] (7,13)-- (5,15);
\draw [line width=2pt,color=ffxfqq] (5,15)-- (4,13);
\draw [line width=2pt,color=ffxfqq] (4,13)-- (3,12);
\draw [line width=2pt,color=ffxfqq] (3,12)-- (2,14);
\draw [line width=2pt,color=ffxfqq] (2,14)-- (1,13);
\draw [line width=2pt,color=ffxfqq] (1,13)-- (1,5);
\draw [line width=2pt,color=ffwwzz] (1,7)-- (1,3);
\draw [line width=2pt,color=ffwwzz] (1,3)-- (2,4);
\draw [line width=2pt,color=ffwwzz] (2,4)-- (3,2);
\draw [line width=2pt,color=ffwwzz] (3,2)-- (4,3);
\draw [line width=2pt,color=ffwwzz] (4,3)-- (5,5);
\draw [line width=2pt,color=ffwwzz] (5,5)-- (5,9);
\draw [line width=2pt,color=ffwwzz] (5,9)-- (4,7);
\draw [line width=2pt,color=ffwwzz] (4,7)-- (3,6);
\draw [line width=2pt,color=ffwwzz] (3,6)-- (2,8);
\draw [line width=2pt,color=ffwwzz] (2,8)-- (1,7);
\draw [color=qqzzqq](2.74863891564831,2.040830744871923) node[anchor=north west] {$\hat{Q_1}$};
\draw [color=ffqqtt](2.526274521006481,6.191632778186068) node[anchor=north west] {$T(\hat{Q_1})$};
\draw [color=zzttqq](6.578247934479812,4.635082015693263) node[anchor=north west] {$P\hat{Q_1}$};
\draw [color=ffxfqq](4.156946748379895,12.664907377759318) node[anchor=north west] {$T(P\hat{Q_1})$};
\draw (1.8097670271605868,7.179918976594198) node[anchor=north west] {$z$};
\draw [line width=2pt,color=zzttqq] (1,1)-- (2,2);
\draw [line width=2pt,color=zzttqq] (2,2)-- (3,0);
\draw [line width=2pt,color=zzttqq] (3,0)-- (4,1);
\draw [line width=2pt,color=zzttqq] (4,1)-- (5,3);
\draw [line width=2pt,color=zzttqq] (1,1)-- (1,9);
\begin{scriptsize}
\draw [fill=ududff] (1.8036677983636795,6.779269720778658) circle (2.5pt);
\end{scriptsize}
\end{tikzpicture}

    \caption{\small This figure shows the situation when $z$ belongs to the upper half of $T(\hat{Q}_1)$. A green cube is a cube $\hat{Q}_1$, the cube bounded by a brown line is a parent of $\hat{Q}_1$, i.e. $P\hat{Q}_1$.}
    \label{fig:enter-label}
\end{figure}
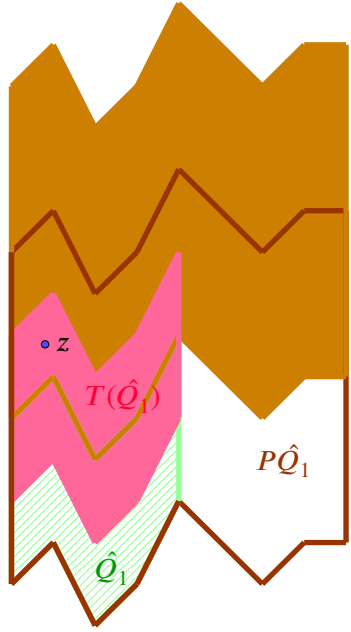

This discussion shows that $\|u-\varphi\|_{\infty}\le (k+1)\varepsilon$.

Notice that 
\begin{equation}\label{Carl-decomp}
\nabla\varphi = (\nabla\varphi_1)\chi_{\hat{Q}_0\setminus{\sc Red}}+(\nabla u)\chi_{\sc Red}+J,
\end{equation}
 where $J$ denotes jumps along boundaries of {\sc Red}. We already proved that $|\nabla\varphi_1|\ud \mathscr{L}^{n+1}$ is a Carleson measure. Therefore, by the definition of $\varphi$ in~\eqref{def-eps-apr}, it remains to prove that on the set {\sc Red}, the measure $|\nabla \varphi|\ud \mathscr{L}^{n+1}=|\nabla u|\ud \mathscr{L}^{n+1}$ is a Carleson measure and that also $J$ gives a Carleson measure. Since $\|u\|_{\infty}\le 1$ and $\|u-\varphi\|_{\infty}\lesssim_k\varepsilon$ it follows that $|J|\lesssim (1+\varepsilon)\sum_{T(\hat{Q}_j) \textnormal{ red}}l(Q_j)^n$. From those it will follow that $|\nabla\varphi|$ defines a Carleson measure.

Our goal amounts to proving the two inequalities (Car1) and (Car2). The first one allows us to handle the $\nabla u$ term in~\eqref{Carl-decomp}, while (Car2) takes care of the $J$ part.
\begin{align*}
&\tag{Car1}\quad \sum_{T(\hat{Q}_j) \textnormal{ red},\, \hat{Q}_j\subset \hat{Q} } \int_{T(\hat{Q}_j)}|\nabla u|\lesssim_{n, L, \eta} \frac{1}{\ep} l(Q)^n, \\
&\tag{Car2}\quad \sum_{T(\hat{Q}_j) \textnormal{ red},\, \hat{Q}_j\subset \hat{Q} } l(Q_j)^n\lesssim_{n, L, \eta} \frac{1}{\ep^2}  l(Q)^n.
\end{align*}
Let us begin with proving (Car2). Let us choose a finite cover by hyperbolic balls centered at points of any given cube $T(\hat{Q}_j)$. Since we are interested now in red cubes, we have by~\eqref{est-Morrey} that for any such ball $B_r$ of radius $r$ from the covering of $T(\hat{Q}_j)$ it holds 
\begin{equation}\label{est-aux-mthm}
\varepsilon^2\le (\osc_{B_r} u)^2\lesssim r^{1-n}\int_{(1+\eta)B_r} |\nabla u|^2,\quad\hbox{for some fixed }\eta\in [0,1).
\end{equation}
Notice that for any point $z\in T(\hat{Q}_j)$ it holds that the distance $\delta_j (z)$ of $z$ to the bottom face of $\hat{Q}_j$, satisfies $\frac12 l(Q_j)\leq \delta_j (z) \leq l(Q_j)$. Moreover, for any hyperbolic ball $B_r$ containing $z$ it holds that $r\leq  \delta_j (z)  \leq Cr$, for some fixed numerical constant $C>0$. Thus, we have that
\[ 
 r\approx \delta_j (z) \approx l(Q_j).
\]
Let us fix one ball $B_r$ from the covering of $T(\hat{Q}_j)$, centered at the point $x_{\hat{Q}_j}+\overline{e_{n+1}}\frac12l(Q_j)$ a center of $T(\hat{Q}_j)$ and such that $(1+\eta)B_r\Subset T(\hat{Q}_j)$. Such a ball can be obtained by similar reasoning as in~\eqref{ball-box} and $r:=\frac{l(Q_j)}{(1+\eta)(C(n,L)+1)^4}$ would suffice.

Therefore, upon multiplying inequality~\eqref{est-aux-mthm} by $l(Q_j)^n$,  we get
\begin{equation}\label{est-Car12}
l(Q_j)^n\le \frac{1}{\varepsilon^2}\int_{T(\hat{Q}_j)}|\nabla u|^2 \delta_j(z),
\end{equation}
as the above constructed ball $B_r$ satisfies $(1+\eta)B_r\subset T(\hat{Q}_j)$. From this and the H\"older inequality together with~\eqref{ball-box} we infer the following estimate 
\begin{align*}
\left(\int_{T(\hat{Q}_j)}|\nabla u|\right)^2 \lesssim_{n,L} \left(\int_{T(\hat{Q}_j)}|\nabla u|^2\right)l(Q_j)^{n+1}\approx\left(\int_{T(\hat{Q}_j)}|\nabla u|^2 \delta_j\right)l(Q_j)^n
\lesssim \frac{1}{\varepsilon^2}\left(\int_{T(\hat{Q}_j)}|\nabla u|^2 \delta_j\right)^2,
\end{align*}
which gives
$$
\int_{T(\hat{Q}_j)}|\nabla u|\lesssim_{n, L}  \frac{1}{\varepsilon}\int_{T(\hat{Q}_j)}|\nabla u|^2 \delta_j.
$$
We now proceed with the first inequality (Car1), as it turns out that proving it, will also complete the proof of (Car2). 
\begin{align}
\sum_{T(\hat{Q}_j) \textnormal{ red},\,\hat{Q}_j\subset \hat{Q} } \int_{T(\hat{Q}_j)}|\nabla u(z)|&\lesssim \sum_{T(\hat{Q}_j) \textnormal{ red},\,\hat{Q}_j\subset \hat{Q} }\frac{1}{\ep}\int_{T(\hat{Q}_j)}|\nabla u(z)|^2 \delta_j(z)\lesssim_{n,L} \frac{1}{\ep}\int_{\hat{Q}}|\nabla u(z)|^2 \delta(z), \label{est2-Car12}
\end{align}
where in the last inequality, by $\delta(z)$ we denote the distance of z point $z$ to the bottom face of  $\hat{Q}$. Moreover, the last inequality holds true, due to observation that although sets $T(\hat{Q}_j)$ may, in general intersect, for different $\hat{Q}_j$, each point in $\hat{Q}$ belongs to at most two sets $T(\hat{Q}_j)$. Thus, the integral on the right-hand side of the last estimate in~\eqref{est2-Car12} may increase at most twice. Finally, similarly to the discussion of estimate~\eqref{lem3-final-est} in the proof of Lemma~\ref{lem2}, we observe that for any point $\om\in \hat{Q}\cap \partial \Om$, i.e. in the bottom face of cube $\hat{Q}$, it holds that
\[
 z\in \Gamma_{\alpha, 0, l(Q)}(\om)\,\Leftrightarrow\, \om \in B(z, (1+\alpha)d(z, \partial \Om)) \cap \partial \Om.
\]
Moreover, notice that for $z=(x,y)\in \hat{Q}$ it holds that $\delta(z) \leq l(Q) \approx_{n,L} y-\phi(x)$. These observations, together with the analogous computations as in~\eqref{lem2-est333} and Proposition~\ref{prop24}, imply that
\begin{align}
\frac{1}{\ep}\int_{\hat{Q}}|\nabla u(z)|^2 \delta(z)\,\ud z &\approx_{n,L}
\frac{1}{\ep}\int_{\hat{Q}}|\nabla u(x,y)|^2 (y-\phi(x))^{1-n} (y-\phi(x))^{n}\,\ud x \ud y \nonumber \\
&\approx_{n,L, \alpha} \frac{1}{\ep} \int_{Q} \left(A_{\alpha, 0, l(Q)} u\right)^2(x)\,\ud x
\lesssim \frac{1}{\ep}(l(Q))^n. \label{est3-Car12}
\end{align}
%
This completes the argument for inequality (Car1) and the proof of (Car2) follows as well, upon combining~\eqref{est-Car12} with \eqref{est2-Car12} and~\eqref{est3-Car12}.

Hence, $|\nabla\varphi|\ud x\ud y$ is a Carleson measure and, therefore, the proof of the $\ep$-approximability of $u$ in $\hat{Q}_0$ is completed.

Notice that in the proof it is not important that we consider a unit cube. Hence, our reasoning gives $\varepsilon$-approximation for any cube $\hat{Q}$ regardless of its side length. To obtain an $\varepsilon$-approximation in set $\Omega$, we follow the approach in the end of the proof of Theorem 1.3 in~\cite[Section 5]{hmm}. Let us choose a point $x_0$ in $\mathbb{R}^n$. Let $Q_k$ be a family of cubes in $\mathbb{R}^n$ such that $x_0$ is a center of each of those cubes and $l(Q_k)=2^k$. Denote by $\varphi_k$ an $\varepsilon$-approximation on set $\hat{Q}_k$. Define
\[
\varphi:=\sum_{k=0}^{\infty} \varphi_k\chi_{\hat{Q}_{k+1}\setminus \hat{Q}_k}.
\] 
One can verify that $\varphi$ is an $\varepsilon$-approximation in $\Omega$ and we leave the details to the reader. Thus, the proof of $\varepsilon$-approximability of $u$ in $\Omega$ is completed.
\end{proof}

\begin{remark}\label{eps-reg}
As observed in several works (e.g.~\cite{hmm, ga, ht}), the regularity of the $\ep$-approximation $\phi$ obtained in the proof above, can be improved to $C^\infty$. Indeed, this follows by Lemmas 3.2 (i) 3.6 and 3.8 in~\cite{ht} and by the standard mollification procedure, see e.g.~\cite[Section 4.2]{eg}.
\end{remark}

\section{Examples of functions satisfying the condition~\eqref{cond-star}}

In this section we provide an example of the class of functions satisfying condition~\eqref{cond-star} and, moreover, that are also $\ep$-approxi\-mable. Furthermore, we also discuss some sufficient conditions giving~\eqref{cond-star} and~\eqref{cond-main}. The importance of the latter condition comes from the fact that it implies \eqref{cond-star}, see Introduction. Our examples illustrate that for some classes of functions, a condition~\eqref{cond-star} more general than~\eqref{cond-main}, suffices for the $\ep$-approxi\-mability to hold.

Let us recall Proposition 5.2 in~\cite{g}. It asserts that if $u$ is a $C^2$-regular function defined in a ball $B\subset \R^{n+1}$ such that $u^{2k}$ is subharmonic in $B$ for some positive integer $k>0$, then $u$ satisfies condition~\eqref{cond-star} in $B$. This observation gives the following wide class of functions satisfying Theorem~\ref{thm-main}.
\begin{prop}\label{prop-31}
Let $u\in C^2$ be nonnegative and subharmonic in an open set $\Om\subset \R^{n+1}$, i.e. $\Delta u\geq 0$. Then $u$ satisfies \eqref{cond-star}.  Moreover, $u$ is $\varepsilon$-approximable in domains $\Om$ as in~\eqref{def-dom}.
\end{prop}
The proof of the first assertion follows by direct computations showing that $\Delta u^{2k}\geq 0$ for any positive integer $k>0$. The second assertion follows immediately from Theorem~\ref{thm-main}, upon noticing that proofs of Lemmas 4.5 and 4.6 in~\cite{g} hold for as well for functions satisfying assumptions of the proposition. Indeed, the scrutiny of the proofs of Lemmas 4.5 and 4.6 reveals that under assumptions of the proposition, it holds that $\Delta u^2=2(|\nabla u|^2+u\Delta u) \geq 2|\nabla u|^2$ and Lemma 4.6 in~\cite{g} follows, if $u$ satisfies~\eqref{cond-star}.

\begin{prop}
  Let $\Om\subset \R^{n+1}$ be an open set and $u\in C^2(\Om)$ be nonnegative and such that  
\[
\Delta|\nabla u|^{\alpha} \geq 0
\]
for any $0<\alpha \leq 2$. Then, conditions~\eqref{cond-star} holds.
\end{prop}
\begin{proof}
 Since $|\nabla u|^{\alpha}$ is $C^2$-regular subharmonic, it satisfies the submean value property on Euclidean balls $B\subset kB\subset \Om$. Hence for any $x,y\in B$ and some point $z\in B$, lying on a line segment joining $x$ and $y$, we have
 \begin{align*}
  |u(x)-u(y)|^{\alpha}&\leq |\nabla u(z)|^{\alpha}|x-y|^{\alpha} \\
  &\leq C(n,\eta) \left(\vint_{(1+\eta)B}|\nabla u|^{\alpha}\right) r^{\alpha}\\
  &\leq  C(n,\eta) \left(\frac{1}{r^{n+1-2}}\int_{(1+\eta)B}|\nabla u|^2\right)^{\frac{\alpha}{2}}\\
  &\leq  C(n,\eta) \left[ \left(\frac{1}{r^{n-1}}\int_{(1+\eta)B}|\nabla u|^2+|u\Delta u|\right)^{\frac12}\right]^{\alpha}.
 \end{align*}
 Therefore, we proved that \eqref{cond-star} holds with $\phi(t)=Ct$ with C depending on $n, \eta$ and the diameter of domain $\Om$.
\end{proof}

Recall that a $C^2$-function satisfies~\eqref{cond-main}, if $|u\Delta u|\le\theta |\nabla u|^2$ for some $0<\theta<1$.  

\begin{prop}\label{prop-cond-star}
  Let $\Om\subset \R^{n+1}$ be an open set and $u$ be a $C^2$-function. If $u\Delta u\geq 0$ in $\Om$, then each of the following conditions implies \eqref{cond-star}: $\Delta \ln u\leq 0$, $\Delta u^{-1}\geq 0$.  Moreover, if $\Delta u^{\alpha}\leq 0$ for some $0<\alpha<1$, then condition \eqref{cond-main} holds with $\theta:=1-\alpha$, and hence, also \eqref{cond-star} holds.
 \end{prop}
\begin{proof}
 The proof of the first assertion is based on the same type of computations and therefore, we will show only argument for the first of the two conditions. We have that, at points in $\Om$ where $u\not=0$, it holds that
\[
  0\geq \Delta \ln u={\rm div}\left(\frac{1}{u}\nabla u\right)=\frac{u\Delta u-|\nabla u|^2}{u^2},
 \]
 and so $|u\Delta u|\le |\nabla u|^2$ holds in $\Om$ (as, if $u=0$ at some point in $\Om$, then this inequality holds trivially). Thus, by the comment following definition of~\eqref{cond-main} in Introduction, condition~\eqref{cond-star} follows from Proposition 5.1 in~\cite{g}, even though $\theta=1$.  By analogy, the following direct calculations give us condition \eqref{cond-main} and show the second assertion:
\[
 0\geq \Delta u^{\alpha}={\rm div}\left(\frac{\alpha}{u^{1-\alpha}}\nabla u\right)=\frac{\alpha}{u^{2-\alpha}}(u\Delta u-(1-\alpha)|\nabla u|^2).
\]
\end{proof}

\section{Appendix: the proof of Proposition~\ref{prop24}}

The reasoning relies on the presentation in~\cite[Section 4]{g} and on a variant of the observation stated on page 261 in~\cite[Exc. 4]{ga}, see also \cite{str}. First, we state the following claim and show how it implies the assertion of Proposition~\ref{prop24}. Then, we prove the claim.

\smallskip
\noindent {\sc Claim.} \emph{Let $f:\Rn \to \R$ be a measurable function and let $c\in(0,\frac12)$ and $\lambda>0$. If for any cube $Q\subset \Rn$ there exists a constant $a_{Q}$ such that 
\[
 \left|\left\{x\in Q: |f(x)-a_{Q}|>\lambda \right\}\right|<c|Q|,
\]
then it holds that
\[
 \left|\left\{x\in Q: |f(x)-a_{Q}|>t \right\}\right|<e^{-c_2t}|Q|,
\]
where $t=3n\lambda$, while $c_2=(3\lambda)^{-1} \ln (4/3)$.
}
\smallskip

Suppose that the claim is proven. Then Lemma 4.2 (i) in~\cite{g} asserts that $A_{\alpha, 0, l}(f)<A_{\alpha', 0, l'} (f_{B_\ep})$, for some $\alpha'>\alpha$, $l'>l$, all $0<\ep<\ep_0(\alpha, L)$  and a nonnegative measurable function $f$. Here, $f_{B_\ep}$ stands for the mean value integral of $f$ over a hyperbolic ball $B_\ep$. Namely:
\[
 f_{B_\ep}:= f_{B_\ep(z,y)}=\vint_{B((z,y), \ep(y-\phi(z)))} f,\quad (z,y)\in \Om.
\]
Thus, Proposition~\ref{prop24} will be proven provided that we show that 
 \[
  \int_{Q} \left(A_{\alpha', 0, l'} |\nabla u|^2_{B_\ep}\right)^2(x)\,\ud x<c |Q| = c (l(Q))^n.
 \]
We find that
\begin{align*}
&\int_{Q} \left(A_{\alpha', 0, l'} |\nabla u|^2_{B_\ep}\right)^2(x)\,\ud x \\
&\,\,=\int_{Q} \left(A_{\alpha', 0, \infty} |\nabla u|^2_{B_\ep}\right)^2(x)-\left(A_{\alpha', l', \infty} |\nabla u|^2_{B_\ep}\right)^2(x)\,\ud x \\
&\,\,\leq \int_{Q} \left(A_{\alpha', 0, \infty} |\nabla u|^2_{B_\ep}\right)^2(x)-\left(A_{\alpha', l', \infty} |\nabla u|^2_{B_\ep}\right)^2(x_Q)\,\ud x+\int_{Q} \left|\left(A_{\alpha', l', \infty} |\nabla u|^2_{B_\ep}\right)^2(x_Q)-\left(A_{\alpha', l', \infty} |\nabla u|^2_{B_\ep}\right)^2(x)\right|\,\ud x. 
\end{align*}
The second integral on the right-hand side is bounded above by $c(\alpha,n, L)|Q|$ in a consequence of applying Lemma 4.3 in~\cite{g}, provided that we know that 
\[
\vint_{B_\ep(z,y)} |\nabla u|^2 \leq \frac{c}{(y-\phi(z))^2},\quad \hbox{for } (z,y)\in \Om.
\]
This condition immediately follows from the Caccioppoli inequality, see Lemma 4.5 in~\cite{g}, with the constant $c=c(n,\theta)\ep^{-2}\|u\|^2_{L^\infty}$. The first integral above we estimate by the Cavalieri formula as follows:
\[
 \int_{Q} \left(A_{\alpha', 0, \infty} |\nabla u|^2_{B_\ep}\right)^2(x)-\left(A_{\alpha', l', \infty} |\nabla u|^2_{B_\ep}\right)^2(x_Q)\,\ud x=\int_{0}^{\infty} \big|\big\{x \in Q: \big|\left(A_{\alpha', 0, \infty} |\nabla u|^2_{B_\ep}\right)^2(x)-a_Q\big| >t \big\}\big|\,\ud t,
\]
where $a_Q:=\left(A_{\alpha', l', \infty} |\nabla u|^2_{B_\ep}\right)^2(x_Q)$. By the reasoning analogous to the proof of  Theorem 4.7 in~\cite{g} and by Corollary 4.8 in~\cite{g}, we know that there exists $t_0$ such that for all $t>t_0$ it holds that
\[
\left|\left\{x\in Q: \big|\left(A_{\alpha', 0, \infty} |\nabla u|^2_{B_\ep}\right)^2(x)-a_Q\big|>t_0 \right\}\right|\leq \frac14|Q|,
\]
as by following the notation of~\cite{g}, see page 217, we may choose $t_0$ such that $\frac{C|Q|}{(t_0-C_1)^b}=\frac14$. By the claim applied to $f:=\left(A_{\alpha', 0, \infty} |\nabla u|^2_{B_\ep}\right)^2$, the latter estimate implies a corresponding one with $e^{-ct}|Q|$, which in turn gives us the assertion of Proposition~\ref{prop24}. 
\smallskip

\noindent It remains to show the above Claim. Without the loss of generality let $c=\frac14$ for the constant $c$ as in the assumptions of the claim. 

First, let $Q_j\subset Q$ denote any cube in the dyadic decomposition of cube $Q$ and set $G_1:=\bigcup Q_j$, the union of all maximal cubes $Q_j$ satisfying the following stopping time condition:
\[ 
|\{x\in Q_j: |f(x)-a_Q|>\lambda\}| \geq \frac13 |Q|.
\]
The family of cubes $G_1$ has the following properties:
\begin{itemize}
\item[(i)] $Q\not \in G_1$, as $c=\frac14$.
\item[(ii)] If $Q_j\in G_1$ and $\hat{Q_j}$ denotes a parent of $Q_j$, i.e. $\hat{Q_j}$ is the minimal cube containing $Q_j$, then as $\hat{Q_j} \not \in G_1$, we have that
\[
 \frac13 |Q_j| \leq |\{x\in Q_j: |f(x)-a_Q|>\lambda\}| \leq  |\{x\in \hat{Q_j}: |f(x)-a_Q|>\lambda\}|<\frac13 |\hat{Q_j}|=\frac23 |Q_j|.
\]
\item[(iii)] If $x\not \in G_1$, then $|f(x)-a_Q|\leq \lambda$ a.e. in $Q$. Indeed, if $x\in Q_k$ for a cube not satisfying the stopping condition, then for a set $E:=\{x\in Q: |f(x)-a_Q|>\lambda \}$, we have that
\[
 \vint_{E} 1_{E}=\frac{|E\cap Q_k|}{|Q_k|}<\frac13,\hbox{ and hence } 1_{E}(x)=0 \hbox{ and }x\not \in E.
\]
The Lebesgue Differentiation Theorem applied to $1_{E}$, a characteristic function of set $E$, gives the property (iii) to hold at a.e. point of $Q$.
\item[(iv)] $\sum_{Q_j \in G_1}|Q_j|\leq \frac34 |Q|$. Indeed, since by the stopping condition $|Q_j|\leq |Q| \leq 3|\{x\in Q_j: |f(x)-a_Q|>\lambda\}|$, we get
\[
\sum_{Q_j \in G_1}|Q_j| \leq \sum_{Q_j \in G_1}  3|\{x\in Q_j: |f(x)-a_Q|>\lambda\}| \leq 3|\{x\in Q: |f(x)-a_Q|>\lambda\}|<\frac34 |Q|.
\]
\end{itemize}
Next, we construct a family of cubes $G_2:=\bigcup Q_k$, consisting of maximal subcubes of cubes in $G_1$ satisfying the following stopping time condition:
\[ 
|\{x\in Q_k: |f(x)-a_{Q_j}|>\lambda\}| \geq \frac13 |Q_j|,\hbox{ for some } Q_j \in G_1.
\]
By property (iv) we get that
\begin{equation}\label{app-aux1}
 \sum_{Q_k \in G_2}|Q_k| \leq \sum_{Q_j \in G_1}\Big(\sum_{Q_k \subset Q_j, Q_k\in G_2}|Q_k|\Big)\leq \sum_{Q_j \in G_1}\frac34 |Q_j|\leq \bigg(\frac34\bigg)^2|Q|.
\end{equation}
Furthermore, for a.e. point $x\not\in G_2$ it holds that
\begin{equation}\label{app-aux2}
|f(x)-a_Q|\leq 3\lambda.
\end{equation}
In order to show~\eqref{app-aux2}, note that if $x\not\in G_2$, then we consider two cases: either (1) $x\not\in G_1$, or (2) $x\in Q_j$ for some $Q_j\in G_1$. In the first case, we have $|f(x)-a_{Q}|\leq \lambda$ by property (ii). In the second case, an argument similar to the one giving property (iii) shows that $|f(x)-a_{Q_j}|<\lambda$ for a.e $x\not \in G_2$.

Next, we show that $|a_Q-a_{Q_j}|<2\lambda$. Define sets
\[
E_1:=\{y\in Q_j: |f(y)-a_{Q}|>\lambda \},\quad E_2:=\{y\in Q_j: |f(y)-a_{Q_j}|>\lambda \}.
\]
Then, by property (ii) it holds that $|E_1|<\frac23 |Q_j|$ and, moreover, by the hypotheses of the claim (recall that we fixed $c=\frac14$) we have $|E_2|<\frac14 |Q_j|$. Furthermore, $(Q_j\setminus E_1)\cap (Q_j\setminus E_2) \not=\emptyset$, as otherwise
\[
 |Q_j|\geq |Q_j\setminus E_1|+|Q_j\setminus E_2|>(1-\frac23)|Q_j|+(1-\frac14)|Q_j|>|Q_j|.
\] 
Therefore, there exists $y\in Q_j$ such that $|f(y)-a_Q|\leq \lambda$ and $|f(y)-a_{Q_j}|\leq \lambda$. This immediately results in the desired estimate
\[
|a_Q-a_{Q_j}|\leq |f(y)-a_Q|+|f(y)-a_{Q_j}|\leq 2\lambda. 
\]
Hence, ~\eqref{app-aux2} follows, as
\[
|f(x)-a_Q|\leq |f(x)-a_{Q_j}|+|a_Q-a_{Q_j}|\leq 3\lambda.
\] 
We iterate the above stopping time procedure and after $n$ steps obtain the family of cubes $G_n$ with the following properties, cf. property (iv) and~\eqref{app-aux1}, \eqref{app-aux2} and :
\[
 (1) \sum_{Q_l\in G_n} |Q_l| \leq \Big(\frac34\Big)^n |Q|,\quad (2)\,\,\, |f(x)-a_Q|<3n\lambda\, \hbox{ for a.e. }x\not\in G_n.
\]
In a consequence, we get that $|\left\{x\in Q: |f(x)-a_Q|>3n\lambda \right\}| \leq \Big(\frac34\Big)^n |Q|$. The latter implies, upon setting $t:=3n\lambda$, the assertion of Claim, as $(3/4)^n=e^{-(\ln 4/3)(3\lambda)^{-1}t}$. This completes the proof of Claim and the proof of Proposition~\ref{prop24} is completed as well.

%

\end{document}